\documentclass[10pt]{article}
\usepackage[utf8]{inputenc}
\usepackage{amsmath,amsthm,amsfonts,mathtools}
\usepackage{graphicx}
\usepackage{color}
\usepackage{multirow}
\usepackage{ulem}
\usepackage{comment}
\usepackage{subcaption}
\usepackage{mathtools}
\usepackage{amssymb}

\newtheorem{theorem}{Theorem}

\newtheorem{remark}{Remark}
\graphicspath{ {./figs/} }

\begin{document}

\title{Implicit-Explicit Scheme with Multiscale Vanka Two-Grid Solver for Heterogeneous Unsaturated Poroelasticity}

\author{
Maria Vasilyeva 
\thanks{Department of Mathematics and Statistics, Texas A\&M University, Corpus Christi, Texas, USA. 
Email: {\tt maria.vasilyeva@tamucc.edu}.}
\and 
Ben S. Southworth  
\thanks{Theoretical Division, Los Alamos National Laboratory, NM, USA. Email: {\tt southworth@lanl.gov}.}
\and 
Yunhui He 
\thanks{Department of Mathematics, University of Houston, Houston, Texas, USA. 
Email: {\tt yhe43@central.uh.edu}.}
\and 
Min Wang 
\thanks{Department of Mathematics, University of Houston, Houston, Texas, USA. 
Email: {\tt mwang55@Central.uh.edu}.}
}

\maketitle

\begin{abstract}
We consider a coupled nonlinear system of equations that describe unsaturated flow in heterogeneous poroelastic media. For the numerical solution, we use a finite element approximation in space and present an efficient multiscale two-grid solver for solving the coupled system of equations. The proposed two-grid solver contains two main parts: (i) accurate coarse grid approximation based on local spectral spaces and (ii) coupled smoothing iterations based on an overlapping multiscale Vanka method. A Vanka smoother and local spectral coarse grids come with significant computational cost in the setup phase. To avoid constructing a new solver for each time step and/or nonlinear iteration, we utilize an implicit-explicit integration scheme in time, where we partition the nonlinear operator as a sum of linear and nonlinear parts. In particular, we construct an implicit linear approximation of the stiff components that remains fixed across all time, while treating the remaining nonlinear residual explicitly. This allows us to construct a robust two-grid solver offline and utilize it for fast and efficient online time integration. A linear stability analysis of the proposed novel coupled scheme is presented based on the representation of the system as a two-step scheme. We show that the careful decomposition of linear and nonlinear parts guarantees a linearly stable scheme. A numerical study is presented for a  nonlinear coupled test problem of unsaturated flow in heterogeneous poroelastic media. We demonstrate the robustness of the two-grid solver, particularly the efficacy of block smoothing compared with simple pointwise smoothing, and illustrate the accuracy and stability of implicit-explicit time integration.
\end{abstract}

\section{Introduction}

Accurately simulating unsaturated flow is essential in hydrology, soil science, reservoir simulation, and environmental engineering. It describes water movement through porous media where both air and water coexist, and is governed by nonlinear partial differential equations (PDEs) combining Darcy’s law and the continuity equation \cite{celia1990general}. The model includes highly nonlinear relationships between moisture content, pressure, and hydraulic conductivity, often described by empirical functions such as the van Genuchten or Brooks–Corey models \cite{van1980closed, brooks1965hydraulic}.  
Coupling mechanics with flow is crucial, as changes in water content affect both permeability and the mechanical properties of the porous matrix, leading to complex interactions \cite{uzuoka2012dynamics, carmeliet2013nonlinear}. 
Numerical modeling of such coupled processes in heterogeneous media is challenging due to strong nonlinearity, multiphysics interactions, and scale dependencies. Advances in numerical methods can directly enhance the accuracy and efficiency of simulations in geotechnical engineering, reservoir modeling, and environmental hydrology \cite{varela2021finite}.
Classical numerical schemes for solving time-dependent problems are based on explicit or implicit time approximation. 
Explicit methods are straightforward in implementation, but can require very small time steps for stability, making them impractical for problems with heterogeneous properties.
Implicit schemes enhance stability for stiff equations, but require {computationally} expensive linear and nonlinear solves. 
An alternative approach involves combining two methods using an Implicit-Explicit (ImEx) time-stepping scheme, which applies an explicit scheme to non-stiff components and an implicit scheme to stiff components that require improved stability  \cite{ascher1995implicit, vabishchevich2013additive}. 
However, even in ImEx schemes, the bulk of computational cost is in solving the implicit equations. 

Recent studies have demonstrated the effectiveness of various multigrid methods on poroelastic models, e.g., \cite{adler2022monolithic,luo2015multigrid}. For challenging problems, robust multigrid methods can require more advanced smoothers than simple pointwise methods such as Jacobi or Gauss-Seidel. This has led to renewed interest in block- or patch-based smoothers, e.g. \cite{farrell2021pcpatch}, along the lines of Schwarz or Vanka methods \cite{vanka1986block}, and related to (spectral) domain decomposition methods \cite{Spillane.2014}. Such methods can be overlapping or nonoverlapping with respect to patches or ``subdomains'' that are locally solved, and additive or multiplicative in nature. Multicolor smoothers, e.g., \cite{kawai2020multiplicative,wang2025multicolor},
combine ideas from additive and multiplicative methods are often used for improved parallelization over multiplicative methods and improved performance over additive methods.

In this work, we {develop both a time-discretization strategy and corresponding two-level solver framework} for unsaturated flow in heterogeneous poroelastic media. Our main contributions are:  
\begin{enumerate}
\item We consider nonlinear poroelastic systems with high spatial variations in conductivity and elastic properties, as well as sharp wetting fronts, requiring fine-scale resolution.
\item We develop a \textit{novel implicit-explicit null (EIN) scheme}, splitting the operator into a linear approximation fixed for all time (treated implicitly) and a nonlinear residual (treated explicitly). It is proven that the proposed splitting ensures unconditional linear stability, and the reuse of the linear implicit component across all time is key to constructing an efficient linear solver.
\item We propose a \textit{{new and} efficient multiscale two-grid solver}, consisting of:
\begin{itemize}
    \item a coarse-grid solver based on \textit{local spectral enrichment} for flow and mechanics;
    \item a \textit{new coupled multiscale Vanka smoother} with overlapping subdomains defined by multiscale basis supports, coarse cells, or extended coarse-cell neighborhoods, and multicolor additive implementations.
\end{itemize}
\item We demonstrate that, despite the additional cost of computing multiscale basis functions and Vanka smoothers, these {components} can be precomputed \emph{offline} and reused, yielding highly efficient simulations for nonlinear unsaturated poroelastic problems.
\end{enumerate}

The paper is organized as follows. In Section \ref{sec:model}, we describe the continuous PDE problem formulation for unsaturated flow in poroelastic media. 
Section \ref{sec:dis} proceeds to introduce a finite element approximation in space, and Section \ref{sec:imex} proposes a new implicit-explicit time integration scheme. We then {provide} a stability analysis {demonstrating that the proposed additive partitioning of the operator leads to an unconditionally stable scheme.} 
In Section \ref{sec:prec}, we propose the construction of a multiscale two-grid solver based on a multiscale Vanka approach, and describe the construction of local spectral multiscale basis functions for the flow and mechanics parts of the system.  
The performance of our proposed two-grid solver with local smoothing iterations for the nonlinear poroelastic problem in heterogeneous media is demonstrated in Section \ref{sec:num}. Conclusions are discussed in Section \ref{sec:con}.

\section{Problem formulation}\label{sec:model}

We consider the unsaturated poroelasticity model that describes the coupled behavior of fluid flow and mechanical deformation in partially saturated porous media. Specifically, we utilize the model from \cite{kim2010sequential,varela2018implementation, varela2021finite} with an additional nonlinearity to account for elastic properties \cite{lu2014power}. 

\subsection{Unsaturated flow}

The Richards equation given below describes an unsaturated flow in porous media, which combines Darcy’s law for variably saturated flow with the mass conservation equation \cite{celia1990general}
\begin{equation}
\label{eq:unsat}    
c_w(x, h_w)\frac{\partial h_w}{\partial t} - \nabla \cdot (k_w(x,h_w) ( \nabla h_w + \nabla \zeta)) = f_w,
\end{equation}
where $\theta$ is the water content, 
$h_w$ is the pressure head, 
$\zeta$ is the elevation head, 
$f_w$ is a fluid source/sink term and 
$k_w$ is the hydraulic conductivity, and $c_w(x, h_w) = \partial \theta / \partial h_w$. 

The nonlinear conductivity coefficient and water content depend on the pressure head $h_w$. The shape of water retention curves can be characterized by the van Genuchten model \cite{van1980closed}:
\begin{equation}
\theta(h_w) = \theta_r + (\theta_s - \theta_r) \left[ 1 + (\beta |h_w|)^{n_{\theta}} \right]^{-m_{\theta}}, \quad h_w < 0,
\label{eq:vg1}
\end{equation}
where $\theta_s$ is the saturated water content, $\theta_r$ is the residual water content, $\beta$ is an empirical parameter related to the inverse of the air entry pressure, and $n_{\theta}$ and $m_{\theta}$ are fitting parameters with $m_{\theta} = 1 - \frac{1}{n_{\theta}}$. The unsaturated hydraulic conductivity function is defined as
\begin{equation}
k_w(x, h_w) = k_s(x) (S_e(h_w))^{\eta} \left[ 1 - \Big(1 - \big(S_e(h_w)\big)^{1/m_{\theta}} \Big)^{m_{\theta}} \right]^2, 
\label{eq:vg2}
\end{equation}
where $k_s$ is the saturated hydraulic conductivity,
$S_e(h_w) = \frac{\theta(h_w) - \theta_r}{\theta_s - \theta_r}$ is the effective saturation, and 
$\eta$ is a pore connectivity parameter.

\subsection{Mechanics}
For the mechanics, we consider the following momentum balance for the porous medium’s displacement under the quasi-static and  small deformations assumptions 
\cite{kim2010sequential, varela2018implementation, varela2021finite}
\begin{equation}\label{eq:elas}
-\nabla \cdot \sigma_T = \rho_b g, 
\end{equation}
where  $$\sigma_T=\sigma - \alpha S p \mathcal{I} ,\quad \sigma(u) =  2 \mu \varepsilon (u) + \lambda (\nabla \cdot u) \mathcal{I}, \quad \varepsilon(u)=0.5 (\nabla u + (\nabla u)^T).$$
Here, $\sigma_T$ is the total stress tensor,
$\sigma$ is the stress tensor, 
$p$ is the water pressure, 
$S$ is the water saturation, 
$\alpha$ is the Biot coefficient, 
$u$ is the displacement vector,
$\varepsilon$ is the strain tensor,
$\rho_b = \big(\phi S \rho_w + (1-\phi)  \rho_s \big)$ is the bulk density,  
$g$ is the gravity vector, 
$\phi$ is the porosity, 
$\rho_w$ is the water density, 
$\rho_s$ is the solid phase density, 
$\mathcal{I}$ is the identity tensor, 
and $\lambda$ and $\mu$ are the Lame's coefficients. 

Additionally, we consider nonlinear elastic properties that describe the effect of changes in water content on the mechanical properties  \cite{lu2014power}
\[
\lambda(x, h_w) = \frac{E(x, h_w) \nu}{(1+\nu)(1-2\nu)},\quad 
\mu(x, h_w) = \frac{E(x, h_w)}{2(1+\nu)},
\]
where $\nu$ is the Poisson’s ratio and  $E(x,h_w)$ is the Young’s modulus given by
\begin{equation}
E(x, h_w) = E_d(x) + (E_w(x) - E_d(x)) (S_e(h_w))^{\zeta}
=  E_d(x) \big(1 + (r - 1) (S_e(h_w))^{\zeta} \big),
\label{eq:vg3}
\end{equation}
where $r_E = E_w(x)/E_d(x)$. 
The elastic properties of the porous matrix are water-dependent, meaning that soils and rocks tend to soften as they absorb water, leading to increased deformations. Here $E_d$ and $E_w$ are Young’s modulus of the dry and wet states, respectively, and $\zeta$ is the empirical fitting parameter \cite{lu2014power}. 

We note that the relationships \eqref{eq:vg1}-\eqref{eq:vg2} are expressed in terms of water content and pressure, therefore we adapt the original van Genuchten model to water pressure and saturation representation by using the following relations
\[
S = \theta / \phi, \quad 
p = \rho_w g h_w.
\]
Here, instead of formulation \eqref{eq:unsat} for $(\theta, h_w)$, we write a mass conservation equation in terms of $(S, p)$ \cite{kim2010sequential,varela2018implementation, varela2021finite}.  A full derivation from physical principles can be found in the Appendix \ref{app1}.

\subsection{Problem setting and assumptions}

We let $\Omega \subset \mathcal{R}^d$ be a computational domain and consider the following coupled nonlinear poroelasticity system in $\Omega$
\begin{equation}
\begin{split}
& c(x, p) \frac{\partial p}{\partial t} 
+ \alpha S(p) \frac{\partial (\nabla \cdot u)}{\partial t} 
- \nabla \cdot \big( \kappa(x, p) \nabla p \big)
= \tilde{f}, 
\quad x \in \Omega, \quad t > 0, 
\\
& -\nabla \cdot \sigma(u)
+ \alpha \nabla \big(S(p) p \big) = \rho_b g, \quad \quad x \in \Omega.
\end{split}
\label{eq:mm3}
\end{equation}
with $\tilde{f} = f + \nabla \cdot \big( \kappa(x, p) \nabla (\rho_w g) \big)$ and 
\[
c(x, p) =  \big( \phi C_w  + (\alpha -\phi) C_s S \big) S + \big(\phi + (\alpha -\phi) C_s  S p \big) S', \quad 
\kappa(x, p) = k_s(x) \frac{k_{rw}(S)}{\mu_w}.
\]
Here $k$ is the intrinsic permeability tensor, $\mu_w$ is the water dynamic viscosity, and $k_{rw}$ is the relative permeability, $C_w$ is the water compressibility, $C_s$ is the compressibility of the solid grains, $S = S(p)$ and $S' = \partial S/\partial p$.

We supplement \eqref{eq:mm3} with the following initial conditions
\begin{equation}
\label{eq:mm4}
p = p^0, \quad u = u^0, \quad x \in \Omega, \quad t = 0, 
\end{equation}
and boundary conditions
\begin{equation}
\begin{split}
&-\kappa \nabla p \cdot n  
= \gamma (p - p_1), 
\quad  x \in \Gamma_p,  \quad 
-\kappa \nabla p \cdot n  
= 0, 
\quad  x \in \partial \Omega / \Gamma_p, \quad t > 0, 
\\
&u = 0, \quad x \in \Gamma_u,\quad 
\sigma \cdot n  
= 0 , \quad x \in \partial \Omega / \Gamma_u, \quad t > 0, 
\end{split}
\label{eq:mm5}
\end{equation}
where $p_1$ is the given pressure on the boundary $\Gamma_p$ and $\gamma$ is the exchange coefficient in Robin-type boundary condition.

We assume the following:
\begin{itemize}
\item $k_s(x)$ and $E_d(x)$ are the intrinsic permeability tensor and drained Young's modulus, respectively, which may vary spatially.  
\item $\mu_w$ is the constant water dynamic viscosity, $k_{rw}(S)$ is the relative permeability, and $S(p)$ is the saturation function, modeled using the nonlinear van Genuchten relation.  
\item $C_w$ is the constant water compressibility, and $C_s$ is the constant solid grain compressibility.
\item $\phi$ is the constant porosity, $\alpha$ is the constant Biot coefficient, $\rho_b$ is the constant bulk density, and $g$ is the gravity vector.
\end{itemize}

For the analysis and numerical solution, we assume that:
\begin{itemize}
\item $c(x,p) > 0$ and $\kappa(x,p)$ is symmetric positive definite for all admissible $x$ and $p$.
\item $S(p)$ is smooth and monotone increasing, with $0 \le S(p) \le 1$.
\item $k_s(x)$ and $E_d(x)$ are bounded and strictly positive.  
\item The domain $\Omega$ is Lipschitz, and the boundary portions $\Gamma_p$ and $\Gamma_u$ have positive measure.
\end{itemize}
Discretizations in space and time are detailed next.

\section{Spatial discretization and semi-implicit time integration}\label{sec:dis}

Here we will detail a finite element method (FEM) discretization with linear basis functions for both pressure and displacement equations. In principle, an additional regularization can be included in the pressure equation to prevent oscillations and stabilize the discretization for certain parameter regimes \cite{rodrigo2016stability, adler2018robust, lee2017parameter, kim2011stability, kolesov2014splitting}, but here we focus on the standard case without additional regularization.

Let
\[
V = \lbrace v \in [H^1(\Omega)]^d:  v = 0 \text{ on } \Gamma_u \rbrace,  \quad 
Q =  H^1(\Omega).
\]
We define a variational formulation of unsaturated poroelasticity problem \eqref{eq:mm3}-\eqref{eq:mm5} as follows: 
find $(p, u) \in Q \times V$ such that 
\begin{equation}
\begin{split}
\int_{\Omega} \alpha S(p) \frac{\partial (\nabla \cdot u)}{\partial t} \ r \ dx
& + \int_{\Omega} c(x, p) \frac{\partial p}{\partial t} \ r \ dx
  + \int_{\Omega} \kappa(x,p) \nabla p \cdot \nabla r \ dx\\
&
+ \int_{\Gamma_p} \gamma \ p \ r \ ds  = \int_{\Omega} \tilde{f} \ r \ dx
 + \int_{\Gamma_p} \gamma \ p_1 \ r \ ds, 
\quad \forall r \in Q, 
\\
\int_{\Omega} \sigma(u) : \varepsilon(v) \, dx
&+  \int_{\Omega} \alpha \nabla \big( S(p) p \big) \cdot \ v \ dx 
  = \int_{\Omega} \rho_b \ g \ v \ dx , 
\quad \forall v \in V.
\end{split}
\label{eq:app1}
\end{equation}

We use a finite element method for spatial discretization. Let $\mathcal{T}^h$ be a partition of the domain $\Omega$, $\mathcal{E}^b$ be a set of all boundary interfaces and $\mathcal{E}^b_p$ be a subset of the boundary interfaces $\mathcal{E}^b_p = \mathcal{E}^b \cap \Gamma_p$. 
Consider the first-order backward Euler implicit scheme with constant time-step size $\tau$. 
Then we obtain the following discrete system in a matrix form for $(p^{n+1}_h, u^{n+1}_h) \in Q_h \times V_h$:
\begin{equation}
\frac{1}{\tau }
\begin{bmatrix}
  M_h^{n+1} & \alpha D_h^{n+1} \\
  0 & 0
\end{bmatrix} 
\begin{bmatrix}
  p^{n+1}_h - p^{n}_h \\
  u^{n+1}_h - u^{n}_h
\end{bmatrix} 
+ 
\begin{bmatrix}
  A_h^{n+1} & 0 \\
 \alpha G_h^{n+1} & K_h^{n+1}
\end{bmatrix} 
\begin{bmatrix}
  p^{n+1}_h \\
  u^{n+1}_h
\end{bmatrix} 
= 
\begin{bmatrix}
  F^{n+1}_{p, h} \\
  F^{n+1}_{u, h}
\end{bmatrix} ,
\label{eq:app2} 
\end{equation}
where $c^{n+1} = c(p^{n+1})$, $k^{n+1} = k(x,p^{n+1})$ and
\[ 
\begin{split}
&A_h^{n+1} = \left[a^{n+1}_{ij} = \int_{\Omega} \kappa^{n+1} \nabla \phi_i \cdot \nabla \phi_j \ dx\right], 
\\ &
K_h^{n+1} = \left[b^{n+1}_{ln} = 
\int_{\Omega} (
2 \mu^{n+1} \varepsilon(\Phi_l) : \varepsilon(\Phi_n) dx + 
\lambda^{n+1} \nabla \cdot \Phi_l \ \nabla \cdot \Phi_n
) dx\right],
\\ &
D_h^{n+1} = \left[d_{il}^{n+1} = \int_{\Omega} S(p^{n+1}) \ \phi_i \nabla \cdot \Phi_l   dx\right],  
\\ &
G_h^{n+1} = \left[g_{li}^{n+1} = \int_{\Omega}  \nabla \big( S(p^{n+1}) \phi_i \big) \cdot  \Phi_l \ dx \right],  
\\ &
F_{p, h}^{n+1} = \left[f^{n+1}_{p,i} = \int_{\Omega} \tilde{f}^{n+1} \ \phi_i \ dx + \int_{\Gamma_p} \gamma p_1  \phi_i \ ds \right], 
\\ &
F_{u, h}^{n+1} = \left[
f_{u,l}^{n+1} = \int_{\Omega} \rho_b^{n+1} \ g \ \Phi_l \ dx\right],  
\quad 
M^{n+1}_h = \left[m^{n+1}_{ij} = \int_{\Omega} c^{n+1} \phi_i \phi_j \ dx \right],  
\end{split}
\]
with
\[
u_h^{n+1} = \sum_{l = 1}^{ N_v} u_l^{n+1} \Phi_l, \quad p_h^{n+1} = \sum_{i = 1}^{N_v} p_i^{n+1} \phi_i,
\] 
where
$\{\Phi_l\}$ are the linear basis functions for displacements, $\{\phi_i\}$ are the linear basis functions for pressure, and $N_v$ is the number of vertices on the grid $\mathcal{T}_h$.

Each implicit time step requires solving a nonlinear system. To linearize, we will use a Picard iteration in the form
\begin{equation}
\begin{split}
\frac{1}{\tau }
\begin{bmatrix}
  M_h^{n+1, (m)} & \alpha D_h^{n+1, (m)} \\
  0 & 0
\end{bmatrix} 
&\begin{bmatrix}
  p^{n+1, (m+1)}_h - p^{n}_h \\
  u^{n+1, (m+1)}_h - u^{n}_h
\end{bmatrix} \\
& + 
\begin{bmatrix}
  A_h^{n+1, (m)} & 0 \\
  \alpha G_h^{n+1, (m)} & K_h^{n+1, (m)}  
\end{bmatrix} 
\begin{bmatrix}
  p^{n+1, (m+1)}_h \\
  u^{n+1, (m+1)}_h
\end{bmatrix} 
= 
\begin{bmatrix}
  F^{n+1, (m)}_{p, h} \\
  F^{n+1, (m)}_{u, h}
\end{bmatrix} ,
\end{split}
\label{eq:app3} 
\end{equation}
where $(m)$ denotes nonlinear iteration, and the process continues until a desired tolerance between successive iterates is obtained, $||p^{n+1, (m+1)}_h - p^{n+1, (m)}_h|| \leq \varkappa_p$ and $||u^{n+1, (m+1)}_h - u^{n+1, (m)}_h|| \leq \varkappa_u$, for tolerance $\varkappa_p$ and $\varkappa_u$.

Using an analogous linearization, we may also linearize \eqref{eq:mm5} on the level of the time integration scheme using the solution from the previous time step to evaluate certain nonlinear quantities: 
\begin{equation}
\frac{1}{\tau }
\begin{bmatrix}
  M_h^{n} & \alpha D_h^{n} \\
  0 & 0
\end{bmatrix} 
\begin{bmatrix}
  p^{n+1}_h - p^{n}_h \\
  u^{n+1}_h - u^{n}_h
\end{bmatrix} 
+ 
\begin{bmatrix}
  A_h^{n} & 0 \\
  \alpha G_h^{n} & K_h^{n}
\end{bmatrix} 
\begin{bmatrix}
  p^{n+1}_h \\
  u^{n+1}_h
\end{bmatrix} 
= 
\begin{bmatrix}
  F^{n}_{p, h} \\
  F^{n}_{u, h}
\end{bmatrix} .
\label{eq:app4} 
\end{equation}
Such an approach is a simple way to reduce computational costs by only solving a single linear system each time step, rather than resolving a fully nonlinear system, albeit typically at a cost of reduced stability and/or larger error constant. This approach balances stability, accuracy, and computational efficiency for the nonlinear poroelastic system, and belongs to the broader class of first-order semi-implicit integrators and nonlinear operator partitions \cite{buvoli2024new}.

\section{Implicit-explicit scheme for time approximation} \label{sec:imex}

To approximate the nonlinear flow problem in time, we now construct an additive EIN splitting, with implicit linear term that is fixed for all time. 
We follow our previous work \cite{kolesov2014splitting}, in which we developed a splitting scheme for the poroelasticity equations to construct stable solvers and decouple the displacement and pressure equations. While the present approach is still based on additive splitting, it is conceptually different. Here, we propose using additive splitting to derive an EIN scheme.

\subsection{Additive scheme}

In the poroelasticity system \eqref{eq:mm3}, the second equation does not contain a time derivative, and the system is in the form of differential algebraic equations. To construct an additive scheme and analyze the stability of the resulting linearized formulation, we begin by differentiating the displacement equation with respect to time:
\begin{equation}
\begin{split}
& c(x, p) \frac{\partial p}{\partial t} 
+ \alpha S(p) \frac{\partial (\nabla \cdot u)}{\partial t} 
- \nabla \cdot (\kappa(x, p) \nabla p )
= \tilde{f}, 
\quad x \in \Omega, \quad t > 0, 
\\
& - \frac{\partial (\nabla \cdot \sigma(u))}{\partial t} 
+ \alpha \frac{\partial \big( S(p) \nabla p \big) }{\partial t} 
= \frac{\partial (\rho_b g) }{\partial t}, 
\quad \quad x \in \Omega, \quad t > 0,
\end{split}
\label{eq:mm3b}
\end{equation}
with initial conditions $p=p^0$ and $u=u^0$ such that
\[
-\nabla \cdot \sigma(u^0) + \alpha \nabla p^0 = \rho_b g, \quad x \in \Omega, \quad t = 0.
\]
Then for a semi-discrete system, we have the following matrix form for $U = (p_h, u_h) \in Q_h \times V_h$:
\begin{equation}
\mathbb{C} \frac{\partial U}{\partial t} + \mathbb{L} U = F,
\label{eq:odeform}
\end{equation}
with 
\[
\mathbb{C} = 
\begin{bmatrix}
  M_h & \alpha D_h \\
  \alpha G_h & K_h
\end{bmatrix}, \quad 
\mathbb{L} = 
\begin{bmatrix}
  A_h & 0 \\
  0 & 0
\end{bmatrix}, \quad 
F = 
\begin{bmatrix}
  F_{p,h} \\
  \delta F_{u,h}
\end{bmatrix}.
\]
The resulting semi-discrete system can be viewed as an initial value problem for a system of ordinary differential equations  \cite{kolesov2014splitting, petzold1982differential}. 

The standard implicit approximation for the problem \eqref{eq:odeform} has the following form
for $U^{n+1} = (p^{n+1}_h, u^{n+1}_h)$
\[
\mathbb{C}^{n} \frac{U^{n+1} - U^n}{\tau} + \mathbb{L}^{n} U^{n+1} = F^n,
\]
where
\[
\mathbb{C}^{n} = 
\begin{bmatrix}
  M_h^{n+1} & \alpha D_h^{n+1} \\
  \alpha G_h^{n+1} & K_h^{n+1}
\end{bmatrix}, \quad 
\mathbb{L}^{n} = 
\begin{bmatrix}
  A_h^{n+1} & 0 \\
  0 & 0
\end{bmatrix}.
\]

To eliminate the need for reassembling the problem operator at each time step, we define the following additive representation of the problem operators
\[
\begin{split}
&\mathbb{C}^n = \mathbb{C}^{(lin)} +\mathbb{C}^{n, (nl)}
=
\begin{bmatrix}
  M_h^{(lin)} & \alpha D_h^{(lin)} \\
  \alpha G_h^{(lin)} & K_h^{(lin)}
\end{bmatrix}
+
\begin{bmatrix}
  M_h^{n, (nl)} & \alpha D_h^{n, (nl)} \\
  \alpha G_h^{n, (nl)} & K_h^{n, (nl)}
\end{bmatrix}, 
\\
&\mathbb{L}^n = \mathbb{L}^{(lin)} +\mathbb{L}^{n, (nl)}
=
\begin{bmatrix}
  A_h^{(lin)} & 0 \\
  0 & 0
\end{bmatrix} 
+
\begin{bmatrix}
  A_h^{n, (nl)} & 0 \\
  0 & 0
\end{bmatrix} .
\end{split}
\]
where $M^{(lin)}$, $D^{(lin)}$, $G^{(lin)}$, $A^{(lin)}$ and $K^{(lin)}$ are linear operators and $M^{n,(nl)}$, $D^{n,(nl)}$, $G^{n,(nl)}$, $A^{n,(nl)}$ and $K^{n,(nl)}$ are nonlinear residual operators, which depend on current solution $(p^n_h, u^n_h)$ and should be updated at each time step, i.e.,
\[
\begin{split}
&M^{n, (nl)}_h = M^n_h - M^{(lin)}_h, \quad 
D^{n, (nl)}_h = D^n_h - D^{(lin)}_h, \quad 
G^{n, (nl)}_h = G^n_h - G^{(lin)}_h,
\\ &
A^{n, (nl)}_h = A^n_h - A^{(lin)}_h,  \quad 
K^{n, (nl)}_h = K^n_h - K^{(lin)}_h.
\end{split}
\] 
Such additive representation separates the nonlinear parts, and by approximating it explicitly, we obtain a linear system without updating the operator at each time layer.

Then, the proposed implicit-explicit scheme can be written as a three-step scheme
\[
\underbrace{
    \mathbb{C}^{(lin)} \frac{U^{n+1} - U^n}{\tau} 
    + \mathbb{L}^{(lin)} U^{n+1}}_{\text{Linear/Implicit}}
+ \underbrace{
    \mathbb{C}^{n, (nl)} \frac{U^{n} - U^{n-1}}{\tau} 
    + \mathbb{L}^{n, (nl)} U^{n}}_{\text{Noninear/Explicit}} 
= F^n,
\]
or
\begin{equation}
\begin{split}
&\underbrace{
\begin{bmatrix}
  M_h^{(lin)} & \alpha D_h^{(lin)} \\
  \alpha G_h^{(lin)} & K_h^{(lin)}
\end{bmatrix} 
\begin{bmatrix}
  p^{n+1}_h - p^{n}_h \\
  u^{n+1}_h - u^{n}_h
\end{bmatrix} 
+ 
\begin{bmatrix}
  \tau A_h^{(lin)} & 0 \\
  0 & 0
\end{bmatrix} 
\begin{bmatrix}
  p^{n+1}_h \\
  u^{n+1}_h
\end{bmatrix}
}_{\text{Linear/Implicit}}
\\&
\quad \quad + 
\underbrace{
\begin{bmatrix}
  M_h^{n, (nl)} & \alpha D_h^{n, (nl)} \\
  \alpha G_h^{n, (nl)} & K_h^{n, (nl)}
\end{bmatrix} 
\begin{bmatrix}
  p^{n}_h - p^{n-1}_h \\
  u^{n}_h - u^{n-1}_h
\end{bmatrix} 
+ 
\begin{bmatrix}
  \tau A_h^{n, (nl)} & 0 \\
  0 & 0
\end{bmatrix} 
\begin{bmatrix}
  p^{n}_h \\
  u^{n}_h
\end{bmatrix}
}_{\text{Noninear/Explicit}}
\\&
\quad \quad = 
\begin{bmatrix}
  \tau F_{p,h}^n \\
  F_{u,h}^n - F_{u,h}^{n-1}
\end{bmatrix}.
\end{split}
\label{eq:imex1} 
\end{equation}
The main question in this implicit-explicit approximation is how to split the operators to ensure the unconditional stability of the scheme. {We will discuss this in the following subsections.}

\subsection{Stability analysis}

Traditionally, splitting schemes for the poroelasticity system have been developed and analyzed by separating the displacement and pressure calculations, allowing each variable to be solved independently. Common approaches include the undrained split, fixed-stress split, drained split, and fixed-strain split methods, whose stability properties have been studied extensively in the literature \cite{kim2011stability, mikelic2013convergence, kolesov2014splitting}. Our proposed approach differs from traditional splitting methods for poroelasticity and is specifically designed to address nonlinearities, where we treat part of the coupling implicitly to avoid these stability problems. 

We use a linear stability analysis of time-dependent PDEs and derive the unconditional stability of the proposed implicit-explicit scheme \cite{vabishchevich2013additive}.  The stability here does not imply nonlinear stability and should be interpreted within the framework of linearized analysis. Next, we show that the scheme is unconditionally stable provided that the explicitly treated residual is bounded by the implicitly treated approximation \cite{anitescu2004implicit, layton2013analysis, bukac2015analysis}. For notation, we let $A_h^{(nl)} \leq (1 - \varrho) A_h^{n, (lin)}$ imply that for $A_h^{(nl)}$ linearized about any state, the operator $(1-\varrho)A_h^{n, (lin)} - A_h^{(nl)}$ is symmetric positive definite (SPD).

\begin{theorem}
\label{t:t2}
Let nonlinear discretization operators be linearized about some valid state, and assume all linearized discretization operators and corresponding linear approximations, e.g. $A_h^{(nl)}$ and $A_h^{n, (lin)}$, are symmetric. Suppose 
\begin{equation}
\begin{split}
A_h^{(nl)} \leq (1 - \varrho) &A_h^{n, (lin)}, \quad 
K_h^{(nl)} \leq (1 - \varrho) K_h^{n, (lin)}, \quad 
M_h^{(nl)} \leq (1 - \varrho) M_h^{n, (lin)}\\
D_h^{(nl)} \leq (1 - \varrho) &D_h^{n, (lin)}, \quad 
G_h^{(nl)}  \leq (1 - \varrho) G_h^{n, (lin)}, 
\end{split}
\label{eq:st4}
\end{equation}
with $0<\varrho<1$. 
Then the solution of the discrete problem \eqref{eq:imex1} is stable. 
\end{theorem}
\begin{proof}
We use a general framework for the three-step scheme \eqref{eq:imex1} and represent blocks of the system as follows \cite{vabishchevich2013additive}
\[
\begin{split}
M_h^{(lin)} \frac{p^{n+1}_h - p^{n}_h}{\tau} 
&+ M_h^{n, (nl)} \frac{p^{n}_h - p^{n-1}_h}{\tau}
+ A_h^{(lin)} p^{n+1}_h 
+ A_h^{n, (nl)} p^{n}_h \\
& = B_{11} \frac{p^{n+1}_h - p^{n-1}_h}{2\tau} 
+ R_{11} (p^{n+1}_h - 2 p^{n}_h + p^{n-1}_h)
+ A_{11} p^{n}_h,
\\
K_h^{(lin)} \frac{u^{n+1}_h - u^{n}_h}{\tau} 
&+ K_h^{n, (nl)} \frac{u^{n}_h - u^{n-1}_h}{\tau}
= B_{22} \frac{u^{n+1}_h - u^{n-1}_h}{2\tau} 
+ R_{22} (u^{n+1}_h - 2 u^{n}_h + u^{n-1}_h),
\\
D_h^{(lin)} \frac{u^{n+1}_h - u^{n}_h}{\tau} 
& + D_h^{n, (nl)} \frac{u^{n}_h - u^{n-1}_h}{\tau}
= B_{12} \frac{u^{n+1}_h - u^{n-1}_h}{2\tau} 
+ R_{12} (u^{n+1}_h - 2 u^{n}_h + u^{n-1}_h),
\\
G_h^{(lin)} \frac{p^{n+1}_h - p^{n}_h}{\tau} 
& + G_h^{n, (nl)} \frac{p^{n}_h - p^{n-1}_h}{\tau}
= B_{21} \frac{p^{n+1}_h - p^{n-1}_h}{2\tau} 
+ R_{21} (p^{n+1}_h - 2 p^{n}_h + p^{n-1}_h)
\end{split}
\]
with 
\[
\begin{split}
&B_{11} = M_h^{(lin)} + M_h^{n, (nl)} + \tau A_h^{(lin)}, \quad 
R_{11} = \frac{1}{2 \tau} \left(M_h^{(lin)} - M_h^{n, (nl)} + \tau A_h^{(lin)}\right), 
\\
&
A_{11} = A_h^{(lin)} + A_h^{n, (nl)} = A_h^n
\\
&
B_{22} = K_h^{(lin)} + K_h^{n, (nl)}, \quad 
R_{22} = \frac{1}{2 \tau} \left( K_h^{(lin)} - K_h^{n, (nl)} \right),
\\
&
B_{12} = \alpha \big(D_h^{(lin)} + D_h^{n, (nl)}\big), \quad 
R_{12} = \frac{\alpha}{2 \tau} \left( D_h^{(lin)} - D_h^{n, (nl)} \right),\\
&
B_{21} = \alpha \big(G_h^{(lin)} + G_h^{n, (nl)}\big), \quad 
R_{21} = \frac{\alpha}{2 \tau} \left( G_h^{(lin)} - G_h^{n, (nl)} \right).
\end{split}
\]

We let $y_1^n = p^n_h$ and $y_2^n = u^n_h$, and  introduce auxiliary variables in order to represent the scheme \eqref{eq:imex1} as a two-step scheme
\[
w_i^n = y_i^n - y_i^{n-1}, \quad v_i^n = \frac{y_i^n + y_i^{n-1}}{2},\quad i=1,2.
\]
For $w_i^n$ and $v_i^n$, we have
\[
\begin{split}
& 
w_i^{n+1} + w_i^{n} = 2 (v_i^{n+1} - v_i^{n}) = y_i^{n+1} - y_i^{n-1}, \quad 
w_i^{n+1} - w_i^{n} = y_i^{n+1} - 2 y_i^{n} + y_i^{n-1}, \\
&
v_i^{n+1} + v_i^{n} = 2 v_i^{n} + (v_i^{n+1}- v_i^{n}), \quad 
y_i^n 
= \frac{1}{2} (v_i^{n+1} + v_i^{n}) - \frac{1}{4} (w_i^{n+1}- w_i^{n}).
\end{split}
\]
It follows that
\[
B_{ij} \frac{y_j^{n+1} - y_j^{n-1}}{2\tau} 
+ R_{ij} (y_j^{n+1} - 2 y_j^{n} + y_j^{n-1}) 
= 
B_{ij} \frac{w_j^{n+1} + w_j^{n}}{2\tau}
+  R_{ij} (w_j^{n+1} - w_j^{n}), 
\]
and
\[
\begin{split}
B_{11} &\frac{y_1^{n+1} - y_1^{n-1}}{2\tau} 
+ R_{11} (y_1^{n+1} - 2 y_1^{n} + y_1^{n-1})
+ A_{11} y_1^{n} \\
&= 
B_{11}  \frac{w_1^{n+1} + w_1^{n}}{2 \tau}
+ R_{11} (w_1^{n+1} - w_1^{n}) 
+ \frac{1}{2} A_{11}(v_1^{n+1} + v_1^{n}) - \frac{1}{4} A_{11}(w_1^{n+1} - w_1^{n}) \\
& = 
B_{11} \frac{w_1^{n+1} + w_1^{n}}{2 \tau}
+ \left( R_{11} - \frac{1}{4} A_{11} \right)  (w_1^{n+1} - w_1^{n}) 
+ \frac{1}{2} A_{11} (v_1^{n+1} + v_1^{n}) \\
& = 
\frac{1}{\tau}B_{11} (v_1^{n+1} - v_1^{n})
+ \left( R_{11} - \frac{1}{4} A_{11} \right)  (w_1^{n+1} - w_1^{n}) 
+ \frac{1}{2} A_{11} (v_1^{n+1} - v_1^{n}) + A_{11} v_1^{n} \\
& = 
\frac{1}{\tau} \left(B_{11} + \frac{\tau}{2} A_{11}\right) (v_1^{n+1} - v_1^{n})
+ Q_{11} (w_1^{n+1} - w_1^{n}) 
+ A_{11} v_1^{n},
\end{split}
\]
with $Q_{11} = \left( R_{11} - \tfrac{1}{4} A_{11} \right)
= \tfrac{1}{2 \tau} \left(M_h^{(lin)} - M_h^{n, (nl)} \right)
+ \tfrac{1}{4}\left(A_h^{(lin)} - A_h^{n, (nl)}\right)$.

Then, we can rewrite \eqref{eq:imex1} as follows
\begin{equation}
\begin{split}
\left(B_{11} +\frac{\tau}{2} A_{11} \right)& \frac{v_1^{n+1} - v_1^{n}}{\tau}
 + Q_{11}(w_1^{n+1} - w_1^{n}) + A_{11} v_1^{n}\\
&
+ B_{12} \frac{v_2^{n+1} - v_2^{n}}{\tau}
+ R_{12} (w_2^{n+1} - w_2^{n})
= F_{p,h}^{n}, 
\\
B_{21} &\frac{v_1^{n+1} - v_1^{n}}{\tau}
+ R_{21} (w_1^{n+1} - w_1^{n})\\
&
+ B_{22} \frac{v_2^{n+1} - v_2^{n}}{\tau}
+ R_{22} (w_2^{n+1} - w_2^{n}) = F_{u,h}^n - F_{u,h}^{n-1}.
\end{split}
\label{eq:st3}
\end{equation}

Additionally, we add equations for $w_1$ and $w_2$ to produce general system with a square matrices \cite{vabishchevich2013additive}
\[
-\tau Q_{ij}  \frac{v_j^{n+1} - v_j^{n}}{\tau} 
+ \frac{\tau}{2} Q_{ij} \frac{w_j^{n+1} - w_j^{n}}{\tau} 
+ Q_{ij} w_i^n
= - Q_{ij} (v_j^{n+1} - v_j^{n}) 
+ \frac{1}{2} Q_{ij} (w_j^{n+1} + w_j^{n}) 
= 0,
\]
where matrices $Q_{ij}$ will be defined later in order to ensure that the resulting general block matrix is skew‑symmetric.

We let $Y = (v_1, w_1, v_2, w_2)$  and rewrite the system \eqref{eq:st3} in the following form  
\begin{equation}    
\mathbb{B} \frac{Y^{n+1} - Y^n}{\tau} + \mathbb{A} Y^n = F,
\label{eq:canonicalform}
\end{equation}
with 
\[
\begin{split}
&\mathbb{B} = 
\begin{bmatrix}
B_{11} +\frac{\tau}{2} A_{11} & \tau Q_{11} & B_{12} & \tau R_{12} 
\\
- \tau Q_{11} & \frac{\tau}{2} Q_{11} & \tau Q_{12} & -\frac{\tau}{2} Q_{12}
\\
B_{21} & \tau R_{21} & B_{22} & \tau R_{22} 
\\
\tau Q_{21} & -\frac{\tau}{2} Q_{21} & -\tau Q_{22} & \frac{\tau}{2} Q_{22}
\end{bmatrix}, 
\\
&\mathbb{A} = 
\begin{bmatrix}
A_{11} & 0 & 0 & 0 \\
0 & Q_{11} & 0 & -Q_{12} \\
0 & 0 & 0 & 0 \\
0 & -Q_{21} & 0 & Q_{22}
\end{bmatrix}, \quad 
F^n = \begin{bmatrix}
F_1^n\\
0\\
F_2^n\\
0
\end{bmatrix},
\end{split}
\]
where $F_1^n = F_{p,h}^{n}$ and $F_2^n = F_{u,h}^n - F_{u,h}^{n-1}$.
Here, we set 
$Q_{22} = R_{22}$,
$Q_{12} = R_{12}$ and 
$Q_{21} = R_{21}$ to preserve a skew-symmetry of the block matrices $\mathbb{B}$ and $\mathbb{A}$ with $R_{12} = -R_{21}^T$

Next, we employ additive representation 
\[
\mathbb{B} = \mathbb{B}^{(0)} + \mathbb{B}^{(1)}, 
\quad 
\mathbb{A} = \mathbb{A}^{(0)} + \mathbb{A}^{(1)}, 
\]
where the superscripts $(0)$ and $(1)$ correspond to the linear and nonlinear parts, respectively, 
then
\[
\mathbb{W} 
=
\left( \mathbb{B} - \frac{\tau}{2} \mathbb{A} \right) 
=
\begin{bmatrix}
B_{11} & \tau Q_{11} & B_{12} & \tau Q_{12} \\
- \tau Q_{11} & 0 & \tau Q_{12} & 0 \\
B_{21} & \tau Q_{21} & B_{22} & \tau Q_{22} \\
\tau Q_{21} & 0 & -\tau Q_{22} & 0
\end{bmatrix}
= \mathbb{W}^{(0)} + \mathbb{W}^{(1)}.
\]
With $A_h = A_h^T$, $M_h = M_h^T$, $K_h = K_h^T$, $D_h = -G_h^T$, and $\check{Y} = (\check{v}_1, \check{w}_1, \check{v}_2, \check{w}_2)$, we have 
\[
\begin{split}
(\mathbb{A} Y , \check{Y}) 
 &= (A_{11} v_1, \check{v}_1) 
+ (Q_{11} w_1, \check{w}_1) - (Q_{12} w_2, \check{w}_1)
- (Q_{21} w_1, \check{w}_2) + (Q_{22} w_2, \check{w}_2)\\
&= (A_{11} v_1, \check{v}_1) 
+ (Q_{11} w_1, \check{w}_1)
+ (Q_{22} w_2, \check{w}_2)
= (\mathbb{A} \check{Y}, Y),
\\
 (\mathbb{W} Y , Y)& = 
(B_{11} v_1, v_1) + \tau (Q_{11} w_1, v_1) +  (B_{12} v_2, v_1) \\
& + \tau (Q_{12}  w_2, v_1)
- \tau (Q_{11} v_1, w_1) + \tau (Q_{12} v_2, w_1) \\
& + (B_{21} v_1, v_2) + \tau (Q_{21} w_1, v_2) + (B_{22} v_2, v_2) \\
& + \tau (Q_{22} w_2, , v_2) 
+ \tau (Q_{21} v_1, w_2) -\tau (Q_{22} v_2, w_2)\\
& = 
(B_{11} v_1, v_1)  + (B_{22} v_2, v_2) 
> 0.
\end{split}
\]
Furthermore, $\mathbb{W}$ is skew-symmetric and we choose an additive representation that gives skew-symmetric matrices $\mathbb{W}^{(0)}>0$ and $\mathbb{W}^{(1)}>0$, where
\[
\begin{split}
&\mathbb{W}^{(0)} = 
\begin{bmatrix}
M_h^{(lin)} + \tau A_h^{(lin)} & 
\frac{1}{2} M_h^{(lin)} + \frac{\tau}{4} A_h^{(lin)} &  
\alpha D_h^{(lin)} & 
\frac{\alpha}{2} D_h^{(lin)}  \\
- \big( \frac{1}{2} M_h^{(lin)} + \frac{\tau}{4} A_h^{(lin)} \big) & 0 & 
\frac{\alpha}{2} D_h^{(lin)} & 0 \\
\alpha G_h^{(lin)} & 
\frac{\alpha}{2} G_h^{(lin)} & 
 K_h^{(lin)} & 
\frac{1}{2} K_h^{(lin)} \\
\frac{\alpha}{2} G_h^{(lin)}  & 0 & 
-\frac{1}{2} K_h^{(lin)} & 0
\end{bmatrix},
\\
& \mathbb{W}^{(1)} = 
\begin{bmatrix}
M_h^{n, (nl)} & 
- \big(\frac{1}{2}  M_h^{n, (nl)} + \frac{\tau}{4} A_h^{n, (nl)}\big) & 
\alpha D_h^{n, (nl)} & 
- \frac{\alpha}{2} D_h^{n, (nl)} 
\\
\big(\frac{1}{2}  M_h^{n, (nl)} + \frac{\tau}{4} A_h^{n, (nl)}\big) & 0 & 
- \frac{\alpha}{2} D_h^{n, (nl)} & 0 
\\
\alpha G_h^{n, (nl)} & 
- \frac{\alpha}{2} G_h^{n, (nl)} & 
 K_h^{n, (nl)} & 
 - \frac{1}{2}K_h^{n, (nl)} \\
- \frac{\alpha}{2} G_h^{n, (nl)} & 0 & 
\frac{1}{2}K_h^{n, (nl)} & 0
\end{bmatrix}.
\end{split}
\]
Then, we have 
\[
\mathbb{B}^{(0)} \frac{Y^{n+1} - Y^n}{\tau} + \mathbb{A}^{(0)} Y^n 
= \tilde{F}, \quad 
\tilde{F} = 
F - \mathbb{B}^{(1)} \frac{Y^{n+1} - Y^n}{\tau} - \mathbb{A}^{(1)} Y^n.
\]
or
\[
\begin{split}
&\mathbb{B}^{(0)} \frac{Y^{n+1} - Y^n}{\tau} + \mathbb{A}^{(0)} Y^n
= 
\mathbb{B}^{(0)} \frac{Y^{n+1} - Y^n}{\tau}
+ \frac{1}{2} \mathbb{A}^{(0)}  ((Y^{n+1} + Y^n) - (Y^{n+1} - Y^n)) 
\\
& \quad \quad  = 
\mathbb{W}^{(0)} \frac{Y^{n+1} - Y^n}{\tau}
+ \frac{1}{2} \mathbb{A}^{(0)} (Y^{n+1} + Y^n)
= \tilde{F}, 
\end{split}
\]
with $\tilde{F} = F - \mathbb{W}^{(1)} \frac{Y^{n+1} - Y^n}{\tau}
- \frac{1}{2} \mathbb{A}^{(1)} (Y^{n+1} + Y^n)$. 

Finally, by multiplying by $v=(Y^{n+1} - Y^n)/\tau$ and using the Cauchy–Schwarz and Young's inequalities, and assuming 
\[
\mathbb{A}^{(1)} \leq (1 - \varrho) \mathbb{A}^{(0)}, \quad
\mathbb{B}^{(1)} \leq (1 - \varrho) \mathbb{B}^{(0)}, \quad 0<\varrho<1,
\]
we obtain
\[
\begin{split}
&\left(  \mathbb{W}^{(0)} \frac{Y^{n+1} - Y^n}{\tau},  \frac{Y^{n+1} - Y^n}{\tau} \right) 
+ \frac{1}{2} \left( \mathbb{A}^{(0)} (Y^{n+1} + Y^n) , \frac{Y^{n+1} - Y^n}{\tau} \right) \\
& \quad = \left(\tilde{F},  \frac{Y^{n+1} - Y^n}{\tau} \right) \leq  
2 \epsilon \left|\left| \frac{Y^{n+1} - Y^n}{\tau} \right|\right|^2_{\mathbb{W}^{(0)}} 
+ \frac{1}{2 \epsilon} ||F||^2_{(\mathbb{W}^{(0)})^{-1}}\\
& \quad 
+ (1 - \varrho) \left(  \mathbb{W}^{(0)} \frac{Y^{n+1} - Y^n}{\tau},  \frac{Y^{n+1} - Y^n}{\tau} \right)  \\
& \quad 
+ (1 - \varrho) \frac{1}{2} \left( \mathbb{A}^{(0)} (Y^{n+1} + Y^n) , \frac{Y^{n+1} - Y^n}{\tau} \right).
\end{split}
\]

Then with $(\mathbb{A}^{(0)} (Y^{n+1} + Y^n) , Y^{n+1} - Y^n) = (\mathbb{A}^{(0)} Y^{n+1} , Y^{n+1}) - (\mathbb{A}^{(0)} Y^{n} , Y^{n})$, we obtain
\[
\begin{split}
&\frac{\varrho }{\tau^2} ||Y^{n+1} - Y^n||^2_{ \mathbb{W}^{(0)}}
+ \frac{\varrho \tau}{2 \tau^2} ||Y^{n+1}||^2_{\mathbb{A}^{(0)}}\\
& \quad 
\leq \frac{\varrho \tau}{2 \tau^2}||Y^{n}||^2_{\mathbb{A}^{(0)}} 
+ \frac{2 \epsilon}{\tau^2} ||Y^{n+1} - Y^n||^2_{\mathbb{W}^{(0)}} 
+ \frac{1}{2 \epsilon} ||F||^2_{(\mathbb{W}^{(0)})^{-1}}.
\end{split}
\]
We set $\varepsilon = \varrho/2$ and get
\[
||Y^{n+1}||^2_{\mathbb{A}^{(0)}}
\leq ||Y^{n}||^2_{\mathbb{A}^{(0)}} 
+ \frac{2 \tau}{\varrho^2} ||F||^2_{(\mathbb{W}^{(0)})^{-1}}
\]
Therefore, we have a stability of the scheme \eqref{eq:canonicalform} in $\mathbb{A}^{(0)}$ norm. 
\end{proof}

\begin{remark}
We emphasize that the above result concerns \textit{linear stability} of the proposed implicit--explicit scheme. Nonlinear stability and convergence are not directly guaranteed by this analysis. However, in practice, the linear stability result is expected to provide useful guidance in regimes where the nonlinear dynamics are well resolved by the discrete time steps, and the linearized operators are bounded as assumed in the statement of Theorem \ref{t:t2}. Numerical experiments in Section~\ref{sec:num} confirm that the scheme remains robust and accurate for the considered range of time steps and problem parameters, suggesting that the linear stability analysis offers practical insight into the method's behavior.  A detailed investigation of nonlinear stability and convergence will be pursued in future work.
\end{remark}

\subsection{Choice of the additive representation and system of linear equation}

Let 
\[
\begin{split}
&\bar{c}(x) = \max_{p \in [p_{min}, p_{max}]} c(x, p), \quad 
\bar{S}(x) = \max_{p \in [p_{min}, p_{max}]} S(x,p), 
\\
&\bar{\kappa}(x) = \max_{p \in [p_{min}, p_{max}]} \kappa(x, p), \quad 
\bar{E}(x) = \max_{p \in [p_{min}, p_{max}]} E(x, p).
\end{split}
\]
Here, for example, $p_{min}$ can be given as the boundary condition $p_1 = p_{min}$ and $p_{max}$ can be given as the initial condition $p_0 = p_{max}$. 

Then, we set
\[
\bar{\sigma} =  2 \bar{\mu}(x) \varepsilon (u) + \bar{\lambda}(x) (\nabla \cdot u) \mathcal{I},
\quad 
\bar{\lambda}(x)  = \frac{\bar{E}(x) \nu}{(1+\nu)(1-2\nu)},\quad 
\bar{\mu}(x) = \frac{\bar{E}(x)}{2(1+\nu)}, 
\]
and 
$\tilde{c}(x, p) = c(x, p) - \bar{c}(x)$, 
$\tilde{\kappa}(x, p) = \kappa(x, p) - \bar{\kappa}(x)$, 
$\tilde{S}(x, p) = S(x, p) - \bar{S}(x)$, 
$\tilde{\mu}(x, p) = \mu(x, p) - \bar{\mu}(x)$, 
$\tilde{\lambda}(x, p) = \lambda(x, p) - \bar{\lambda}(x)$.

Finally, we have  the following operators in the  implicit-explicit scheme \eqref{eq:imex1}
\begin{itemize}
\item \textit{Linear/Implicit}:
\[ 
\begin{split}
&A_h^{(lin)} = \Big[ a^{(lin)}_{ij} = \int_{\Omega} \bar{\kappa} \nabla \phi_i \cdot \nabla \phi_j \ dx \Big], 
\quad 
M^{(lin)}_h = \Big[m^{(lin)}_{ij} = \int_{\Omega} \bar{c} \ \phi_i \phi_j \ dx \Big],  
\\ &
K_h^{(lin)} = \Big[b^{(lin)}_{ln}=
\int_{\Omega} (
2 \bar{\mu} \varepsilon(\Phi_l) : \varepsilon(\Phi_n) dx + 
\bar{\lambda} \nabla \cdot \Phi_l \ \nabla \cdot \Phi_n
) dx \Big],
\\ &
D_h^{(lin)} = \Big[d_{il}^{(lin)} = \int_{\Omega} \bar{S} \ \phi_i \nabla \cdot \Phi_l   dx \Big],  
\quad 
G_h^{(lin)} = \Big[g_{li}^{(lin)}  = \int_{\Omega}  \nabla \big( \bar{S} \phi_i \big) \cdot  \Phi_l \ dx \Big],  
\end{split}
\]
\item \textit{Nonlinear/Explicit}
\[ 
\begin{split}
&A_h^{n, (nl)} = \Big[a^{n, (nl)}_{ij} = \int_{\Omega} \tilde{\kappa}^n \nabla \phi_i \cdot \nabla \phi_j \ dx \Big],  \quad 
M^{n, (nl)}_h = \Big[m^{n, (nl)}_{ij}  = \int_{\Omega} \tilde{c}^n \phi_i \phi_j \ dx \Big], 
\\ &
K_h^{n, (nl)} = \Big[b^{n, (nl)}_{ln} = 
\int_{\Omega} (
2 \tilde{\mu}^{n} \varepsilon(\Phi_l) : \varepsilon(\Phi_n) dx + 
\tilde{\lambda}^{n} \nabla \cdot \Phi_l \ \nabla \cdot \Phi_n
) dx \Big],
\\ &
D_h^{n, (nl)} = \Big[d_{il}^{n, (nl)}  = \int_{\Omega} \tilde{S}^n \ \phi_i \nabla \cdot \Phi_l   dx \Big], 
\quad 
G_h^{n, (nl)} = \Big[g_{li}^{n, (nl)}  = \int_{\Omega}  \nabla \big( \tilde{S}^n \phi_i \big) \cdot  \Phi_l \ dx \Big].
\end{split}
\]
\end{itemize}

The implementation of the scheme \eqref{eq:imex1} at each time iteration can be written as follows 
\begin{equation}
\begin{bmatrix}
  M_h^{(lin)} + \tau A_h^{(lin)} & \alpha D_h^{(lin)} \\
  \alpha G_h^{(lin)} & K_h^{(lin)}
\end{bmatrix} 
\begin{bmatrix}
  p^{n+1}_h \\
  u^{n+1}_h
\end{bmatrix}  
= \begin{bmatrix}
  b^n_{p,h} \\
  b^n_{u,h}
\end{bmatrix} 
\label{eq:imex2} 
\end{equation}
with 
\[
\begin{split}
\begin{bmatrix}
  b^n_{p,h} \\
  b^n_{u,h}
\end{bmatrix} 
& = 
\begin{bmatrix}
 \tau F_{p, h} \\
  F_{u,h}^n - F_{u,h}^{n-1}
\end{bmatrix}
+ 
\begin{bmatrix}
  M_h^{(lin)} & \alpha D_h^{(lin)} \\
  \alpha G_h^{(lin)} & K_h^{(lin)}
\end{bmatrix} 
\begin{bmatrix}
  p^{n}_h \\
  u^{n}_h
\end{bmatrix} \\
&- 
\begin{bmatrix}
  M_h^{n, (nl)} & \alpha D_h^{n, (nl)} \\
  \alpha G_h^{n, (nl)} & K_h^{n, (nl)}
\end{bmatrix} 
\begin{bmatrix}
  p^{n}_h - p^{n-1}_h \\
  u^{n}_h - u^{n-1}_h
\end{bmatrix} 
 -
\begin{bmatrix}
 \tau A_h^{n, (nl)} & 0 \\
 0 & 0
\end{bmatrix} 
\begin{bmatrix}
  p^{n}_h \\
  u^{n}_h
\end{bmatrix}.
\end{split}
\]
Here, at each iteration, we have a linear system, where the matrix does not depend on the current solution. Therefore, the matrix is constructed only once, and by pre-inverting/pre-factorizing or pre-initializing a solver, we can significantly accelerate the simulations.

\section{Two-grid solver}\label{sec:prec}

This section outlines the construction of the two-grid solver for the unsaturated poroelasticity problem, focusing on its two main components: (1)
the construction of the coarse grid approximation and (2) the design of the coupled smoothing iterations. 

Using the proposed implicit-explicit scheme \eqref{eq:imex2}, we consider the following linear system on each time iteration for $y^{n+1}_h = (p^{n+1}_h,  u^{n+1}_h)$
\begin{equation}
L_h^{(lin)} y^{n+1}_h = b^n_h,
\label{eq:tg1} 
\end{equation}
where $L_h^{(lin)}$ is defined using the linear part of the operators, 
\[
L_h^{(lin)} =
\begin{bmatrix}
  M_h^{(lin)} + \tau A_h^{(lin)} & \alpha D_h^{(lin)} \\
  \alpha G_h^{(lin)} & K_h^{(lin)}
\end{bmatrix}, \quad 
b^n_h = \begin{bmatrix}
  b^n_{p,h} \\
  b^n_{u,h}
\end{bmatrix} 
\]

In this work, we consider the following two-grid algorithm for a given initial guess $y^{(0)}$
\begin{enumerate}
\item Coarse-grid correction:
\begin{enumerate}
\item Restriction:  $r_H = P^T (b^n_h - L_h^{(lin)} y^{(0)})$.
\item Coarse-grid solution: $L_H^{(lin)} e_H = r_H$ with $L_H^{(lin)} = P^T L_h^{(lin)} P$
\item Interpolation and update: $y^{(1)} = y^{(0)} + P e_H$.
\end{enumerate}
\item Post-smoothing: $y^{(2)} = y^{(1)} + S_h^{-1} (b^n_h - L_h^{(lin)} y^{(1)})$ and  $y^{n+1}_h = y^{(2)}$.
\end{enumerate}
Here,  $S_h \in \mathbb{R}^{N_h \times N_h}$ is a smoothing operator and $L_H^{(lin)} \in \mathbb{R}^{N_H \times N_H}$ is the Galerkin coarse grid matrix constructed using the prolongation operator $P \in \mathbb{R}^{N_h \times N_H}$ and the restriction operator $R = P^T \in \mathbb{R}^{N_H \times N_h}$. 
We note that in this setting, we only use post-smoothing iterations \cite{chung2015residual, tyrylgin2021online}.

\subsection{Coarse grid solver}

To construct a coarse grid solver, we follow a similar approach as considered in geometric multigrid method and apply a continuous Galerkin approximation on the coarse grid.  We define the local domain $\omega_l$ as a combination of several coarse grid cells $K_i$ that share the same coarse grid node ($l = 1,...,N_v^H$, and $N_v^H$ is the number of coarse grid vertices). We denote the coarse grid with $\mathcal{T}^H = \cup_{i=1}^{N_c^H} K_i$ (see Figure \ref{meshc}).

\begin{figure}[h!]
\centering
\includegraphics[width=0.4\linewidth]{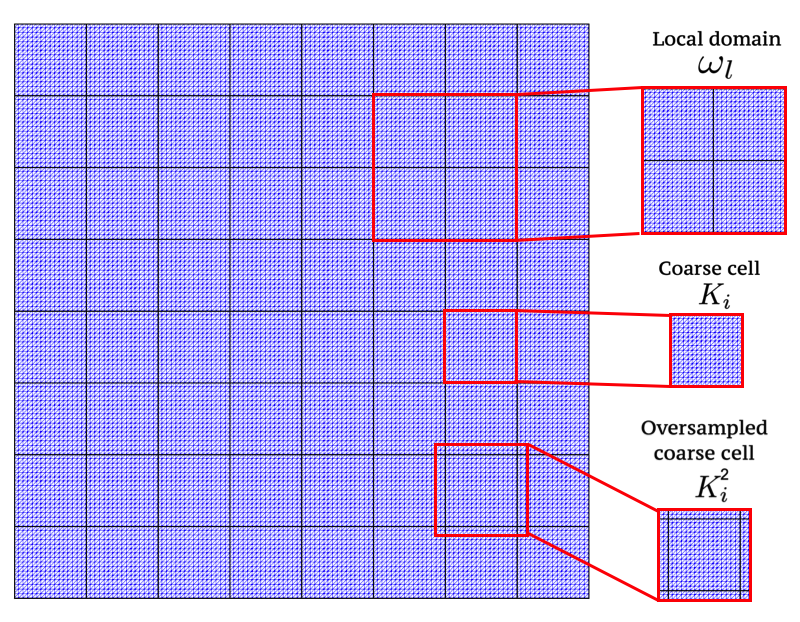}
\caption{Fine grid (blue color), coarse grid (black color), local domain $\omega_l$, coarse cell $K_i$ and oversampled coarse cell $K_i^o$ for $o=2$ with local fine grid resolution.}
\label{meshc}
\end{figure}
We note that the multiscale space is constructed in a preprocessing stage for a given heterogeneity setting \cite{brown2016generalized, akkutlu2018multiscale}. 

\textit{Multiscale basis functions for pressure.}
We solve the following local spectral problem for $(\lambda_{p,j}, \phi^{\omega_l}_j)$ in the local domain $\omega_l$ with a fine-scale resolution
\begin{equation}
\label{eq:evp-p}
A_{h}^{(lin),\omega_l} \phi^{\omega_l}_j = \lambda_{p,j} S_{h}^{A,\omega_l} \phi^{\omega_l}_j,
\end{equation}
where 
\[
\begin{split}
& A_{h}^{(lin),\omega_l} = [a_{ij}], \quad 
a_{ij} 
= \int_{\omega_l} \bar{\kappa}(x) \nabla \phi_i \cdot \nabla \phi_j dx,
\\ &
S_{h}^{A,\omega_l} = [s^A_{ij}], \quad 
s^A_{ij}
= \int_{\omega_l} \bar{\kappa}(x) \phi_i \phi_j dx.
\end{split}
\]
We choose eigenvectors $\phi^{\omega_l}_j$ ($j = 1,..,M^{l,p}$) corresponding to the smallest $M^{l,p}$ eigenvalues in $\omega_l$ and multiply by linear partition of unity functions $\chi^l$ to obtain conforming basis functions
\[
Q_{ms} = \text{span} \lbrace 
\psi^{\omega_l}_j, \, l = 1,...,N^H_v, \, j = 1,...,M^{l,p}
\rbrace,
\]
where $\psi^{\omega_l}_j = \chi^l \phi^{\omega_l}_j$ and $\chi^l$ is the linear partition of unity function in $\omega_l$.

\textit{Multiscale basis functions for displacements.}
We solve the following local spectral problem for $(\lambda_{u,j}, \Phi^{\omega_l}_j)$ in the local domain with a fine-scale resolution
\begin{equation}
\label{eq:evp-u}
K_{h}^{(lin),\omega_l}  \Phi^{\omega_l}_j = \lambda_{u,j} S_{h}^{K, \omega_l} \Phi^{\omega_l}_j,
\end{equation}
where 
\[
\begin{split}
&K_{h}^{(lin),\omega_l} = [b_{ij}], \quad  
b_{ij} = \int_{\omega_l} \bar{\sigma}(\Phi_i) : \varepsilon(\Phi_j) dx,
\\ &
S_{h}^{K,\omega_l} = [s^K_{ij}], \quad 
s^K_{ij} = \int_{\omega_l} (\bar{\lambda} + 2 \bar{\mu}) \Phi_i \Phi_j dx. 
\end{split}
\]
We choose eigenvectors $\Phi^{\omega_l}_j$ ($j = 1,..,M^{l,u}$) corresponding to the smallest $M^{l,u}$ eigenvalues and multiply by linear partition of unity functions for obtaining conforming basis functions
\[
V_{ms} = \text{span} \lbrace 
\Psi^{\omega_l}_j , \, l = 1,...,N^H_v, \, j = 1,...,M^{l,u}
\rbrace,
\]
where $\Psi^{\omega_l}_j = \chi^l \Phi^{\omega_l}_j$.

\textit{Coarse grid system. }
Using constructed multiscale basis functions for pressure and displacements, we define the prolongation matrix as follows
\begin{equation}
P = \begin{bmatrix}
P_p & 0 \\
0 & P_u
\end{bmatrix}
\end{equation}
where 
\begin{equation}
\label{eq:projmatrix}
\begin{split}
& P_p = (
\psi^{\omega_1}_1 \ldots \psi^{\omega_1}_{M^{1,p}} 
\ldots 
\psi^{\omega_{N^H_v}}_1 \ldots \psi^{\omega_{N_v^H}}_{M^{N^H_v,p}}
), \\
& P_u = (
\Psi^{\omega_1}_1 \ldots \Psi^{\omega_1}_{M^{1,u}} 
\ldots 
\Psi^{\omega_{N^H_v}}_1 \ldots 
\Psi^{\omega_{N^H_v}}_{M^{N^H_v,u}}
).
\end{split}
\end{equation}
We note that for discretization in the local domain $\omega_l$, we use the same approximation method as used in the global problem \eqref{eq:app2} and construct a map from local degress of freedoms (DOFs) to global ones as the prolongation operator. 

Therefore, we have
\begin{equation}
\begin{split}
L_H^{(lin)} 
= P^T L_h^{(lin)} P
& = 
\begin{bmatrix}
P_p^T & 0 \\
0 & P_u^T
\end{bmatrix}
\begin{bmatrix}
  M_h^{(lin)} + \tau A_h^{(lin)} & \alpha D_h^{(lin)} \\
  \alpha G_h^{(lin)} & K_h^{(lin)}
\end{bmatrix}
\begin{bmatrix}
P_p & 0 \\
0 & P_u
\end{bmatrix}\\
& = \begin{bmatrix}
  P_p^T (M_h^{(lin)} + \tau A_h^{(lin)}) P_p & \alpha P_p^T D_h^{(lin)} P_u \\
  \alpha P_u^T G_h^{(lin)} P_p & P_u^T K_h^{(lin)} P_u
\end{bmatrix}\\
&= \begin{bmatrix}
  M_H^{(lin)} + \tau A_H^{(lin)} & \alpha D_H^{(lin)} \\
  \alpha G_H^{(lin)} & K_H^{(lin)}
\end{bmatrix}.
\end{split}
\label{eq:tg3} 
\end{equation}
with 
\[
\begin{split}
&M_H^{(lin)} = P_p^T M_h^{(lin)} P_p, \quad 
A_H^{(lin)} = P_p^T A_h^{(lin)} P_p, \quad 
K_H^{(lin)} = P_u^T K_h^{(lin)} P_u, 
\\
& D_H^{(lin)} = P_p^T D_h^{(lin)} P_u, \quad 
G_H^{(lin)} = P_u^T G_h^{(lin)} P_p.
\end{split}
\]
We use a direct solver to solve the coarse-scale problem, assuming aggressive coarsening so that the resulting coarse system is small and can be inverted efficiently. The total number of coarse-scale unknowns is $N_H = N_H^p + N_H^u$, where $N_H^p = \sum_{l=1}^{N_v^H} (1 + M^{l,p})$ and $N_H^u = \sum_{l=1}^{N_v^H} (1 + M^{l,u})$. For smoothing iterations, we consider the classical pointwise smoothers Jacobi and Gauss-Seidel, and further propose a coupled blockwise smoother.

\subsection{Smoothing iterations}

We start with defining subdomains based on the coarse grid that effectively balances computational efficiency with inter-subdomain communication, and reduces errors in a coupled way for our poroelasticity problem. 
In our approach, a coarse grid is aligned with coarse facets, such that $K_i \cap K_j = E_{ij}$. Then the computational domain $\Omega$ is divided into $N_c^H$ subdomains $\{K_i\}_{i=1}^{N_c^H}$ with minimal overlap consisting of nodes located on the interfaces $E_{ij}$. We then increase overlap considering the oversampled local domain $K_i^o$, which is constructed based on the coarse cell $K_i$ by adding $o$-layers of fine-mesh width $h$ ($o=1,2$). The combination of overlapping Schwarz smoothers of this form with local spectral coarse grids is highly related to domain decomposition methods based on generalized eigenvalue problems in the overlap (GenEO) \cite{Spillane.2014}. 
Last, we consider the case with the largest overlap define by the support of coarse grid basis functions $\omega_l$. In Figure \ref{meshc}, we depict local domains $\omega_l$ and $K_i^{o}$ for $o=0$ and $o=2$. In the context of multigrid smoothing, we update the approximation in each iteration based on the local residual over a subdomain corresponding to either the coarse grid cell $\iota_i = K_i$ or the support of the multiscale basis function $\iota_i = \omega_i$.

\begin{figure}[h!]
\centering
  \begin{minipage}[b]{0.45\textwidth}
\centering
\includegraphics[width=1\linewidth]{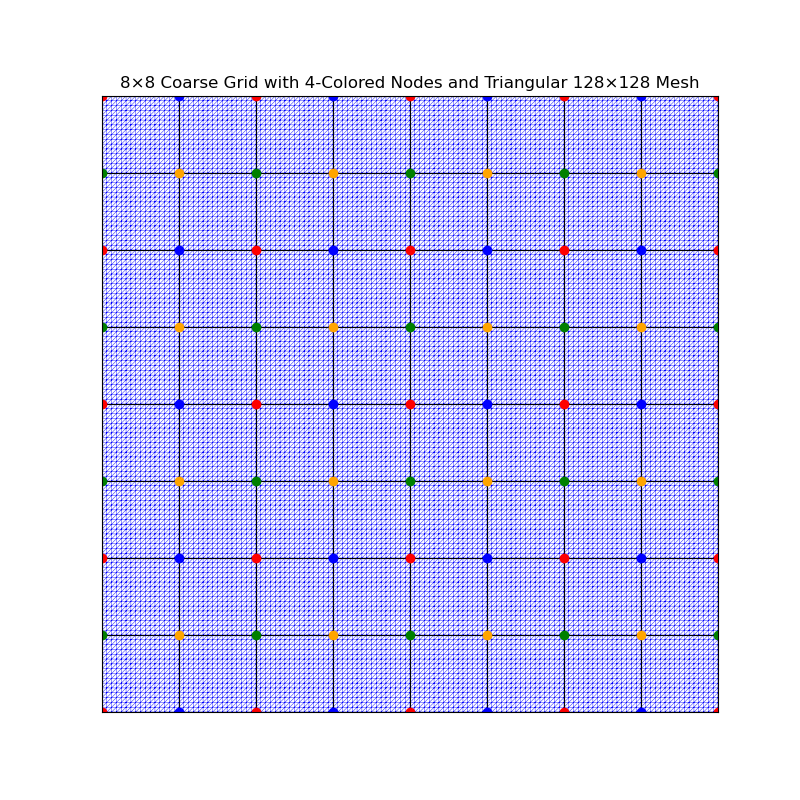}
\end{minipage}
\begin{minipage}[b]{0.45\textwidth}
\centering
\includegraphics[width=1\linewidth]{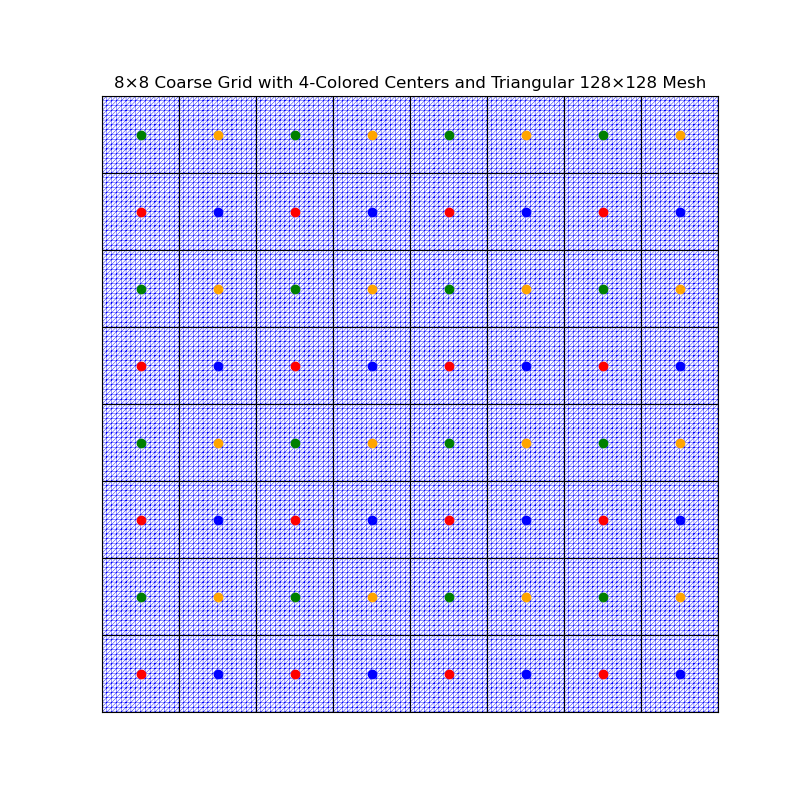}
\end{minipage}
\caption{Coarse grid and groups of subdomains $J^k$ ($k=1,2,3,4$). Left panel: $\iota_i = \omega_i$. Right panel: $\iota_i = K_i^o$.}
\label{blockcg}
\end{figure}

Let $\iota_i$ be the subdomain of global domain $\Omega = \cup_{i=1}^{N_{\iota}^H} \iota_i$, where $\iota_i = K_i$, $K_i^o$ or $\omega_i$ with $N^{\iota}_H = N^H_c$ for $\iota_i = K_i$, $K_i^o$ and $N^{\iota}_H = N^H_v$ for $\iota_i = \omega_i$ ($N^H_c$ and $N^H_v$ are the number of coarse grid cells and vertices, respectively). 
Then, we formulate the smoothing iterations as a Schwarz-type method, where we solve the following local coupled poroelasticity problems within each local domain $\iota_i$
\[
L_h^{(lin), i} w^n_i = r_i^n,
\]
where $L_h^{(lin), i}$ is the restriction of the global system matrix $L_h^{(lin)}$  to the subdomain $K_i$, and  $r^n_i$ is the restriction of the global residual $r^n$ to the local domain $K_i$, i.e.,
\[
L_h^{(lin), i} = R_i L_h^{(lin)} R_i^T, \quad 
r^n_i = R_i r^n,
\]
where $R_i$ is the restriction operator that maps global degrees of freedom to local degrees of freedom in the $i$th subdomain. Due to overlaps at coarse cell interfaces or on bigger overlap, the global correction can be expressed as a weighted average at interface points
\[
w^n = \sum_{i=1}^{N_{\iota}^H} R_i^T W_i w^n_i,
\]
where $W_i$ is the weighting matrix to account for the overlapping degrees of freedom, with each diagonal entry equal to the reciprocal of the number of patches that containing the corresponding degree of freedom \cite{greif2023closed, saberi2022restricted}. 

Then, the smoothing iterations can be represented as follows
\[
y^{(2)} = y^{(1)} + S^{-1}_h r^n = y^{(1)} + w^n, 
\]
with $w^n = S^{-1}_h r^n$,  $r^n = b_h^n - L_h^{(lin)} y^{(1)}$, 
where $y^{(1)}$ is the given approximation on the fine  grid solver 
and the global smoothing operator $S^{-1}_h$ is defined as
\[
S^{-1}_h = \sum_{i=1}^{N_{\iota}^H} R_i^T W_i (L_h^{(lin),i})^{-1} R_i.
\]
and the smoothing iterations can therefore be formulated as
\begin{equation}
y^{(2)}  = y^{(1)} 
+ \sum_{i=1}^{N_{\iota}^H} R_i^T W_i (L_h^{(lin),i})^{-1} R_i (b_h^n - L_h^{(lin)} y^{(1)}),
\label{eq:sm1}
\end{equation}
which corresponds to the additive type of smoothing iterations in subdomains.

We focus on the structured uniform coarse grid and define four groups of subdomains $J^k = \cup_{i \in I_k} \iota_i$ ($k=1,2,3,4$) based on indexing groups $I_k$ for a colored patch relaxation. This adds a level of sequential solving to the relaxation, as patches in group $k+1$ will relax with an updated residual following relaxation on patches in group $k$. In Figure \ref{blockcg}, we plot a coarse grid (black lines) and groups of subdomains for $\iota_i = \omega_i$ (maximal overlap) and $\iota_i = K_i^o$ (minimal overlap). We use four different colors to represent these groups in Figure \ref{blockcg}, with each colored point indicating the center of a subdomain (either a coarse node or a coarse cell) belonging to a specific index set $I_k$. 

In the additive approach \eqref{eq:sm1}, updates are performed in each subdomain independently in a parallel manner, making it highly scalable with $N_{\iota}^H$ independent local solves. We can also formulate the smoothing iterations in a multiplicative way, 
where updates are applied sequentially in each subgroup using the latest available solution from the previous groups, leading to faster convergence than the additive process. Moreover, the sequential approach is scalable in each group of subdomains, where calculations within the group can be done in a parallel with $N_{\iota}^H/4$ independent local solves.
The sequential iterations over groups can be represented as a fractional stepping \cite{st2007optimized}
\begin{equation}
y^{(1+k/M)}
= y^{(1+(k-1)/M)}
+ \sum_{i \in I^k} R_i^T W_i (L_h^{(lin),i})^{-1} R_i 
(b_h^n - L_h^{(lin)} y^{(1+(k-1)/M)} ),
\end{equation}
where $k=1,...,M$ and $M = 4$ is the number of subdomain groups (colors). 

\section{Numerical results}\label{sec:num}

We consider two numerical tests to evaluate the proposed splitting algorithm for the unsaturated poroelasticity problem in a heterogeneous domain:

\textit{Test 1 (splitting scheme and two-grid solver).}
We consider a fixed heterogeneous medium with prescribed contrast and model parameters in a two-dimensional setting.
(\textit{Test 1a}) First, we assess the accuracy of the proposed implicit--explicit splitting scheme by comparing it with an implicit time-stepping method employing nonlinear iterations.
(\textit{Test 1b}) Next, we investigate the performance of the proposed multiscale two-grid solver with coupled smoothing iterations on three different fine-scale grids.

\textit{Test 2 (model parameters and computational cost).}
We study the performance of the proposed two-grid method under varying heterogeneity and model parameters. 
(\textit{Test 2a}) First, we examine the effects of heterogeneity, contrast, nonlinearity, and boundary conditions on the performance of the two-grid method.
(\textit{Test 2b}) Next, we analyze in detail the offline (precomputation) and online computational costs of the proposed two-grid method for both two- and three-dimensional problems. In addition, we present results obtained without additional spectral basis functions, corresponding to the first constant eigenvector. In this case, the resulting coarse space consists of standard linear basis functions, making the method similar to the geometric multigrid like approach. However, due to the aggressive coarsening strategy and the use of only two levels, this approach is not a classical geometric multigrid method.

For the two-grid solver, we consider
\begin{itemize}
\item  implicit--explicit scheme with $N_t = 20$  time steps and fixed coarse grid solver on an $8^d$ coarse grid, where $d$ denotes the spatial dimension ($d=2,3$);

\item the number of local spectral basis functions $M = 1, 2, 4, 8$, where local spectral basis functions for pressure and displacement are precomputed for a given heterogeneous field based on the linear setting described above;

\item smoothing iterations are chosen as follows: 
(i) Gauss--Seidel (GS) iterations from the {pyamg} library \cite{bell2023pyamg}, 
(ii) the proposed coupled smoothing iterations relative to the overlapping domain decomposition method, with local domains corresponding to the support of the multiscale basis functions $\omega_i$ (V), and 
(iii) coupled smoothing iterations with local domains corresponding to the coarse cells $K_i^0 = K_i$ (VK) and overlapped coarse cells including a few fine-cell layers $K_i^o$ with $o=1$ and $2$ (VK1 and VK2);

\item Simulations are performed without pre-smoothing iterations, and the number of post-smoothing iterations is varied from 1 to 3;

\end{itemize}
The implementation is carried out in Python using the SciPy sparse library (version~1.16.2) to solve the linear systems at each time step \cite{virtanen2020scipy}. Finite element matrices are constructed using first-order continuous polynomial spaces for both the pressure and displacement variables. The finite element implementation is based on the FEniCS library~\cite{logg2012automated} (version~2019.1.0).
To solve the eigenvalue problems arising in the construction of the multiscale basis functions, we use slepc4py, the Python interface to the SLEPc library~\cite{hernandez2005slepc} (version~3.23.3). For the iterative solver, we set the relative tolerance to $10^{-9}$ and the maximum number of iterations to~500.
Simulations are run on a MacBook Pro with an Apple M2 Max chip and 32GB RAM.

\subsection{Test 1}

We consider a two-dimensional formulation on the computational domain $\Omega = [0,L]^2$ with $L = 10~\mathrm{m}$.
The heterogeneous saturated hydraulic conductivity $k_s(x)$~[$\mathrm{m}^2$] and the heterogeneous Young’s modulus in the dry state $E_d(x)$~[Pa] are shown in Figure~\ref{results1} (first and second panels), respectively.

\begin{figure}[h!]
\begin{center}
\includegraphics[width=1\linewidth]{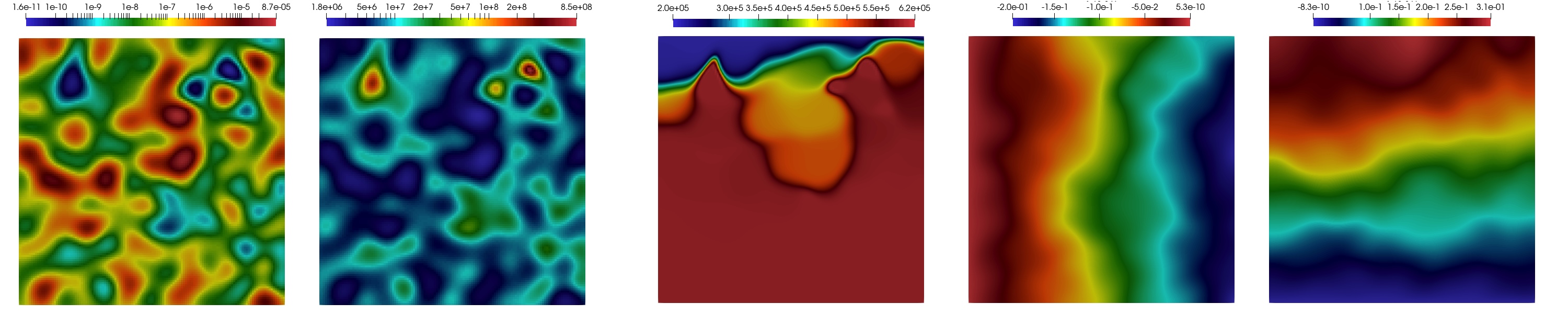}
\end{center}
\caption{Test 1. Heterogeneous coefficients $k_s(x)$[m$^2$] and $E_d(x)$[Pa], pressure [Pa] and displacements $x$ and $y$  components [cm] at final time (from left to right).}
\label{results1}
\end{figure}

In Figure \ref{results1}, we plot a reference solution for an unsaturated poroelastic model with $N_t = 20$ on $128 \times 128$ grid. We depict the pressure and the $x$ and $y$ components of the displacements at final time $t_{20}$. We can observe a strong influence of property heterogeneity on both pressure and displacements.

The model parameters are specified as follows:
\begin{itemize}
\item The van Genuchten model \eqref{eq:vg1}--\eqref{eq:vg2} is employed to describe the hydraulic nonlinearity, with $\nu = 0.37$, $\theta_r = 0.03$, $\theta_s = 0.45$, $\beta = 0.01$ [1/cm], and $n_{\theta} = 1.6$, where $m_{\theta} = 1 - 1/n_{\theta}$.
These parameters correspond to silt-type soils \cite{siltecho2015use, lu2014power}. For the nonlinear elastic model \eqref{eq:vg3}, we set $r_E = E_d / E_w = 2$ and $\zeta = 1.5$ \cite{lu2014power}.
\item The porosity is set to $\phi = 0.45$  and the Biot coefficient is set to $\alpha = 0.2$. The compressibilities of water and solid grains are taken as $C_w = C_s = 10^{-10}$ [1/Pa]. We assume water density $\rho_w = 1000$ [kg/m$^3$] and gravitational acceleration $g = 9.8$ [m/s$^2$].

\item For the flow problem, a Robin-type boundary condition is imposed on the top boundary with
$p_1 = 202{,}860$ [Pa] and $\gamma = 10^6$, while no-flux boundary conditions are applied on the remaining boundaries.
For the mechanical problem, zero displacement is enforced in the $x$-direction on the left boundary and in the $y$-direction on the bottom boundary, whereas the remaining boundaries are assumed to be stress-free.

\item The initial pore pressure is set to $p_0 = 602,700$ [Pa]. 
Simulations are carried out over the time interval $t \in [0, T_{\max}]$ with
$T_{\max} = 2$~days. The time step size is defined as $\tau = T_{\max} / N_t$.
\end{itemize}

We consider the following numerical tests:
\begin{itemize}
\item \textit{Test 1a (accuracy of the splitting scheme).}
We assess the accuracy of the proposed implicit--explicit splitting scheme by comparing it with a fully implicit time-stepping method employing nonlinear iterations.
The comparison is performed for varying numbers of time steps, $N_t = 10, 20, 40,$ and $80$, on a fixed, uniformly refined $128 \times 128$ fine grid.

\item \textit{Test 1b (performance of the two-grid solver).}
We investigate the performance of the proposed two-grid solver on  $N \times N$ grids composed of triangular elements, with $N = 64, 128,$ and $256$.
Linear basis functions are used for both pressure and displacement variables.
The corresponding total numbers of degrees of freedom are $\mathrm{DOF}_h = 3 (N+1)^2 = 12{,}657$, $49{,}923$, and $198{,}147$, respectively.
\end{itemize}

\subsubsection{Test 1a (implicit-explicit scheme)}  

First, we compare the performance of the three time integration schemes: (i) an implicit time-stepping scheme combined with Picard iterations to resolve the nonlinearity \eqref{eq:app3} (Im), a semi-implicit scheme, linearized about the previous solution \eqref{eq:app4} (sIm), and (iii) our proposed implicit-explicit scheme for linearization as in \eqref{eq:imex1}  (ImEx). 

\begin{figure}[h!]
\centering
\includegraphics[width=1\linewidth]{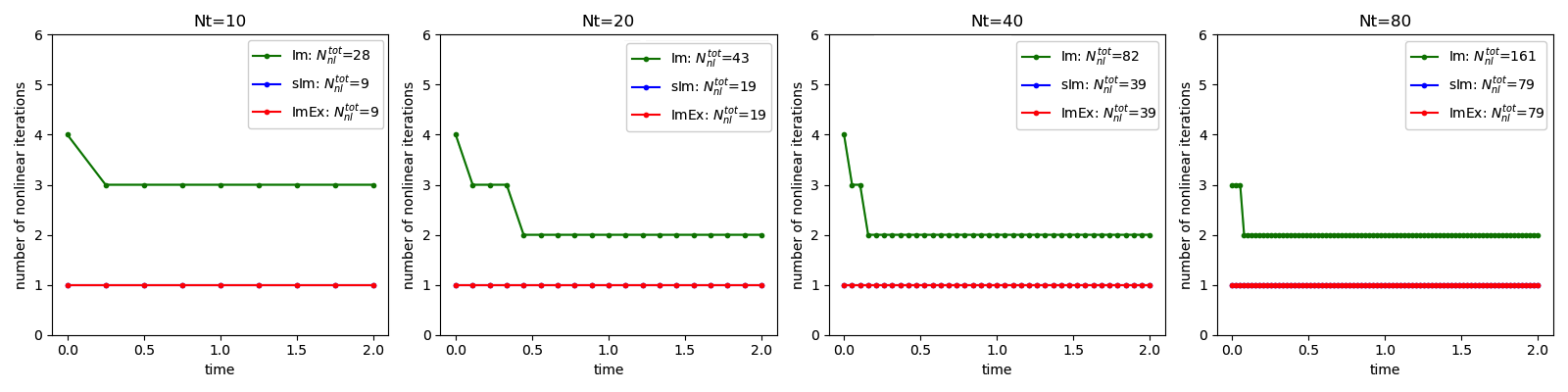}
\caption{Number of nonlinear iterations for implicit scheme with Picard iterations. 
Shown in legend total number of nonlinear iterations ($N^{tot}_{nl}$). }
\label{nonlin1}
\end{figure}

\begin{figure}[h!]
\centering
\includegraphics[width=0.6\linewidth]{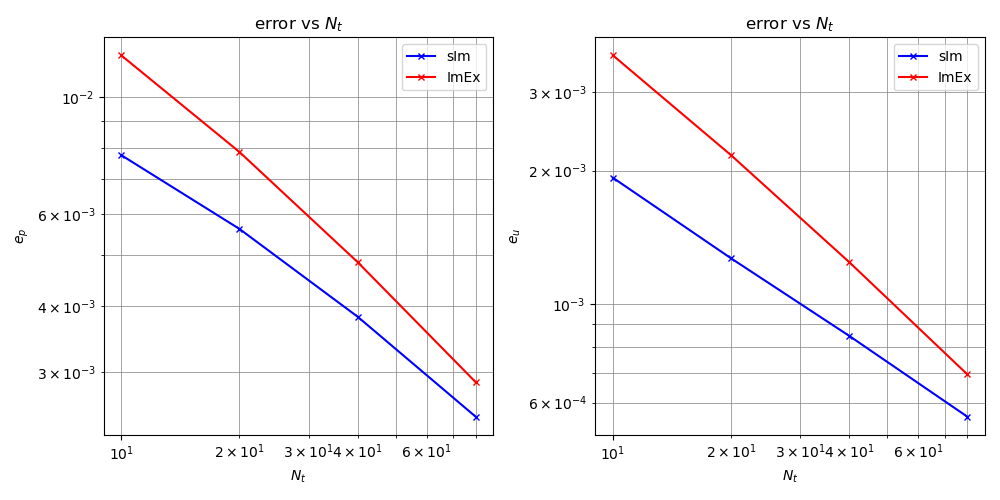}
\caption{Number of time steps ($N_t$) vs relative error for pressure $e_p$ (left) and displacements $e_u$ (right) at final time.}
\label{nonlin3}
\end{figure}

To compare two approaches for time integration, we calculate the relative error in $L_2$ norm
\[
e_p (p_{\text{ref}}, p) = \frac{||p_{\text{ref}} - p||}{||p_{\text{ref}}||}, \quad 
e_u (u_{\text{ref}}, u) = \frac{||u_{\text{ref}} - u||}{||u_{\text{ref}}||}, \quad 
||p||^2 = (p, p),
\]
where $p_{\text{ref}}$ and $u_{\text{ref}}$ are the reference solution using the implicit scheme and Picard iterations. 
In nonlinear iterations, we set $N_{nl} = 10$ as a maximum number of nonlinear iterations and iterate till $e_p(p^{n+1, (m+1)}_h, p^{n+1, (m)}_h)\leq 0.1$ \% and $e_u(u^{n+1, (m+1)}_h, u^{n+1, (m)}_h)\leq 0.1$ \%. 

In Figure \ref{nonlin1}, we plot the number of nonlinear iterations for an implicit scheme (Im) with Picard iterations in green color. We investigate the influence of the number of time steps $N_t = 10, 20, 40$. For all tested $N_t$, more Picard iterations are required in the initial time steps, eventually stabilizing to two iterations per step. Larger $N_t$ leads to a quicker reduction to two nonlinear iterations. We note that in the linearized scheme (sIm--blue, ImEx ---red), only one iteration is performed per time step. 
In Figure \ref{nonlin3}, we represent errors at a final time for different $N_t$ showing stable first-order convergence of both schemes. 
We observe almost the same errors for pressure and displacement for both linearization methods (sIm and ImEx), which decreased for a larger number of time steps as expected.

\subsubsection{Test 1b (multiscale two-grid solver with coupled smoothing iterations)} 

Next, we consider the performance of the proposed two-grid solver for an unsaturated poroelasticity problem in heterogeneous media. 
In the proposed approach with the additive splitting of the linear and nonlinear part, we have the advantage of pre-initializing a solver for a given heterogeneous properties. This includes the construction of a projection/prolongation operator, the construction of a coarse scale matrix based on the operator's linear part, and initializing coupled smoother that includes pre-inverting local matrices in the smoothing algorithm based on the linear part of the operator. 

\begin{figure}[h!]
\centering
\includegraphics[width=1\linewidth]{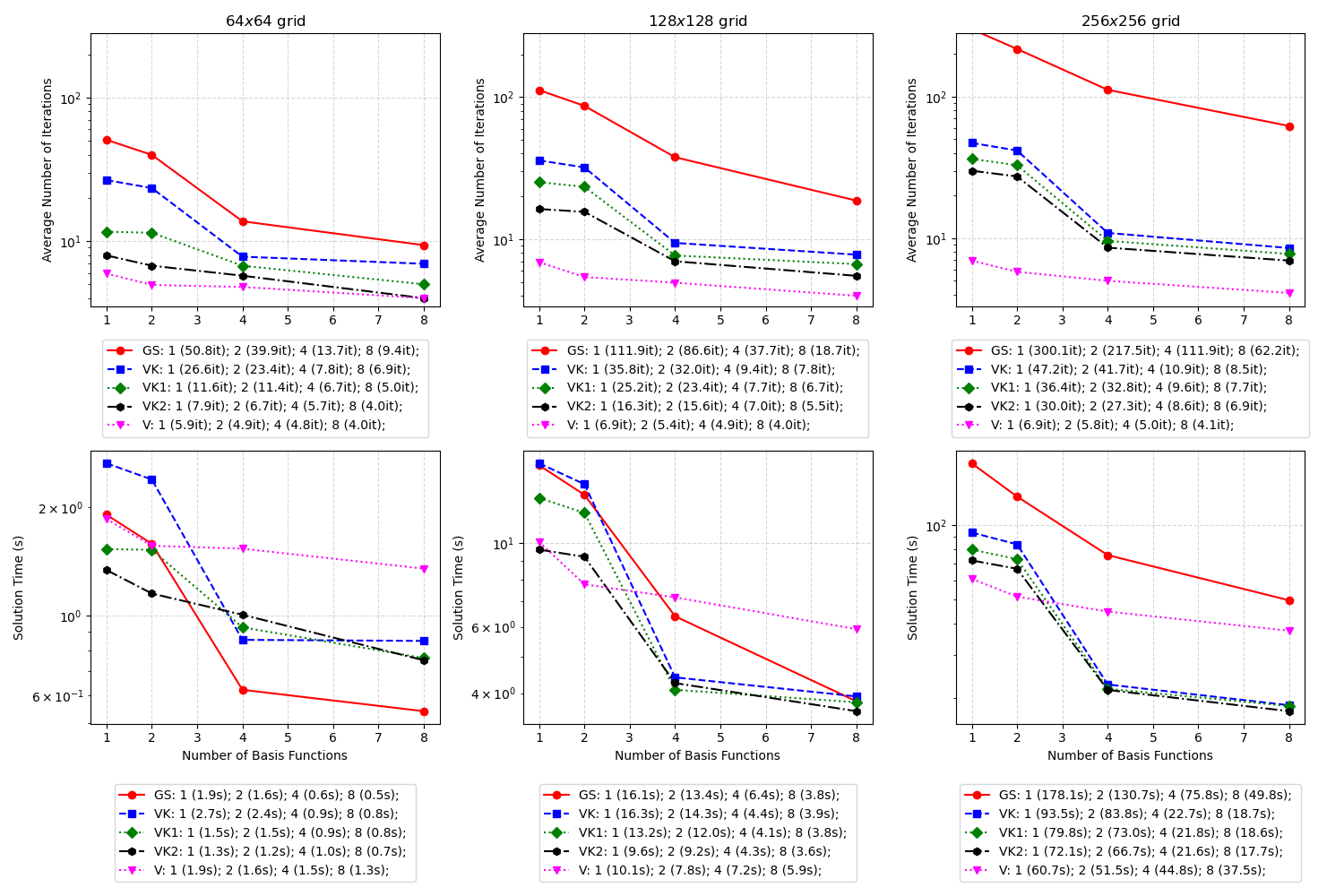}
\caption{Average number of iterations of linear solver (first row) and solve time (second row) for a multiscale two-grid solver with different types of smoothers for $64 \times 64$, $128 \times 128$ and $256 \times 256$ grids (from left to right). GS: Gauss-Seidel. V: $H/2$ overlapped subdomains $\omega_i$. VK, VK1 and VK2: minimal (interface), $h$ and $2h$ overlapped subdomains $K_i^o$ with $o=0,1,2$.}
\label{msiter1}
\end{figure}

\begin{figure}[h!]
\centering
\includegraphics[width=1\linewidth]{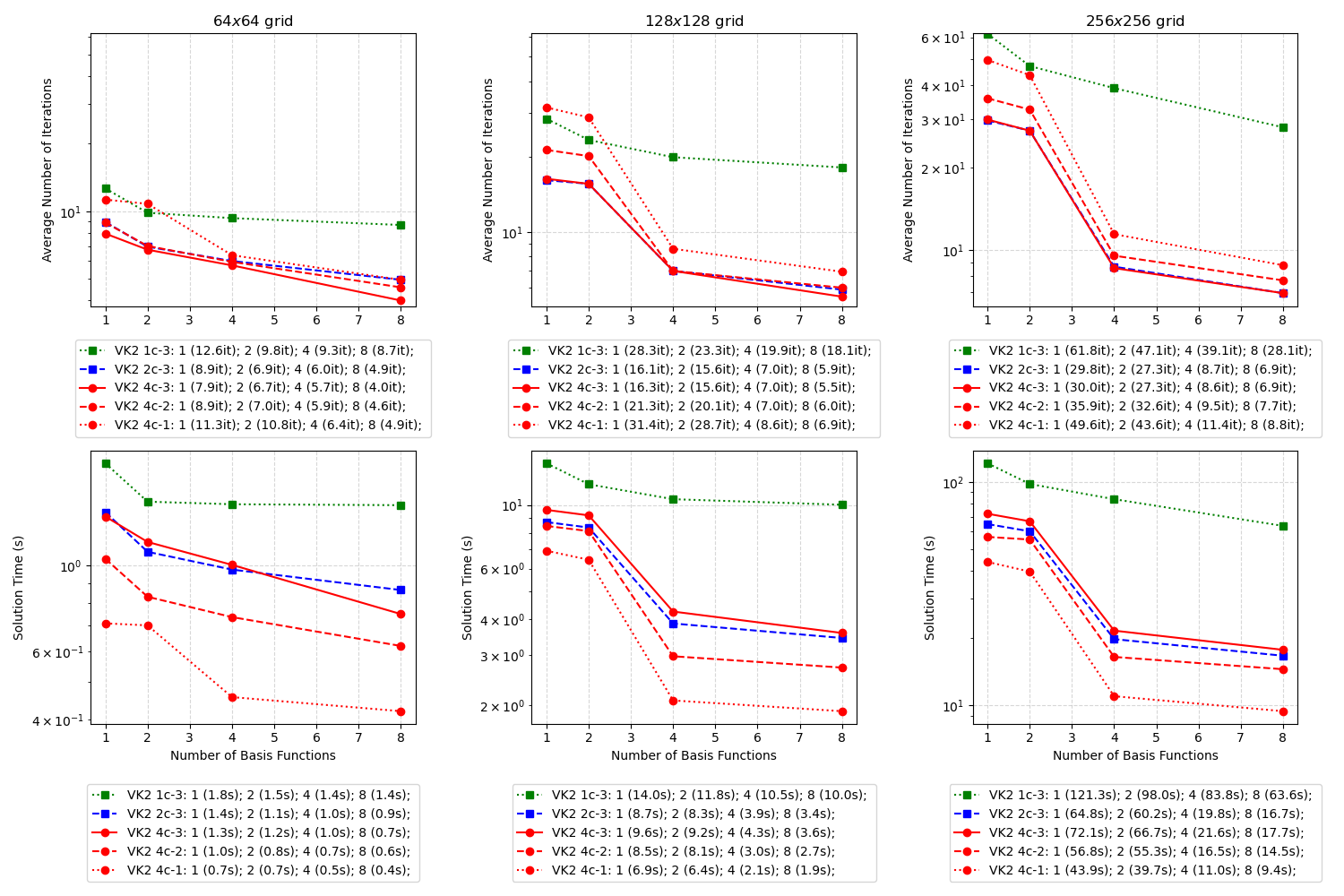}
\caption{Performance of coupled smoother (VK2) with different multicoloring strategies and number of smoothing iterations (log-scale). Three grids $64 \times 64$, $128 \times 128$ and  $256 \times 256$ (from left to right).  Top: average iteration counts of linear solver. Bottom: corresponding solve times.}
\label{msiter2}
\end{figure}

We present results with three smoothing iterations for all types of smoothers, and the proposed coupled smoothing iterations are performed using multiscale Vanka smoother with four colors and 3 post-smoothing iterations.  
In Figure \ref{msiter1}, we investigate the convergence and the corresponding solve time. We show an average number of iterations of linear solver and solve time (without ``offline'' time of generating system and initializing coarse solver) for a multiscale two-grid solver with different types of smoothers for the poroelasticity problem on three grids ($N =64, 128$ and $256$). 
We observe a clear dependence of the solution time on number of basis functions, where increasing basis function 1 to 8 almost always leads to improved wallclock times. When we precompute basis functions, the overhead cost of 8 vs. 1 is very small, so the best practice here is using more basis functions. All of the overlapping Vanka approaches reduce iteration counts by an order of magnitude or more on the finest grid compared with Gauss-Seidel, and demonstrate scalability in fine-mesh size, whereas Gauss-Seidel results in increased iteration count with mesh refinement. This is to some extent to be expected, because the size of subdomains increases in size relative to fine-scale mesh with mesh refinement, but when the subdomain inverses can  be precomputed and used across all times, such an approach does not raise significant scalability issues. In almost all cases, coarse cell-based Vanka with overlap of one or two cells is the best in terms of wallclock time. Using Vanka over the full subdomains $\{\omega_i\}$ decreases iterations vs. Vanka based on coarse cells, but is working too hard in the sense that it increases wallclock time compared with cell-based Vanka.

In Figure \ref{msiter2}, we present a visual comparative performance study of the VK2 smoother across different multicoloring strategies and coarse grid resolutions.  The first row represents the average number of linear iterations versus the number of multiscale basis functions for three grids (from left to right). The corresponding solve times (in seconds) are in the second row. 
Across all grids, the number of iterations decreases significantly as the number of multiscale basis functions $M$ increases due to the better accuracy of the coarse space. We see that for $M = 1$, all smoother configurations require more than 10 iterations. When we take a larger $M \geq 4$, most smoother configurations converge in about 3–4 iterations, indicating that the solver becomes increasingly robust as the coarse space captures more global solution behavior. Among smoother variants, the 4-color smoother configurations consistently converge with 2-4 iterations. 
Varying the number of smoothing iterations from 1 to 3 improves convergence speed but increases solve time.
We observe that multicoloring is essential in the proposed approach, where 4-color smoothers consistently reduce both the number of iterations and wall-lock time.

\subsection{Test 2}

We consider both two- and three-dimensional formulations on the computational domain $\Omega = [0,L]^2$ and $[0,L]^3$ with $L = 1~\mathrm{m}$. 
We investigate the performance of the proposed two-grid method under varying levels of heterogeneity, different model parameters and boundary conditions, through the following numerical tests:
\begin{itemize}
\item \textit{Test 2a (model parameters).}
We study the influence of heterogeneity, contrast, nonlinearity, and boundary conditions on the performance of the proposed two-grid method for a two-dimensional problem using a $128 \times 128$ fine grid ($N_h=49,923$) with $N_t = 20$ time steps (see Figures~\ref{results2a}, \ref{results2b}, and~\ref{results2c}).

\item \textit{Test 2b (computational cost).}
We analyze both the offline (precomputation) and online computational costs of the proposed two-grid method for two- and three-dimensional problems (see Figures~\ref{results2a} and~\ref{results3}). In addition, we consider a coarse space constructed using standard linear basis functions only (i.e., without additional spectral basis functions).  We consider $128 \times 128$ and $256 \times 256$ two-dimensional problems in \textit{Test 2a} and three-dimensional problem on $32 \times 32 \times 32$ grid in \textit{Test 2b}. 
The corresponding total numbers of degrees of freedom are $\mathrm{DOF}_h = 3 (N+1)^2 = 49,923$, $198,147$, and $143,748$, respectively.
\end{itemize}

\begin{figure}[h!]
\centering
\includegraphics[width=1\linewidth]{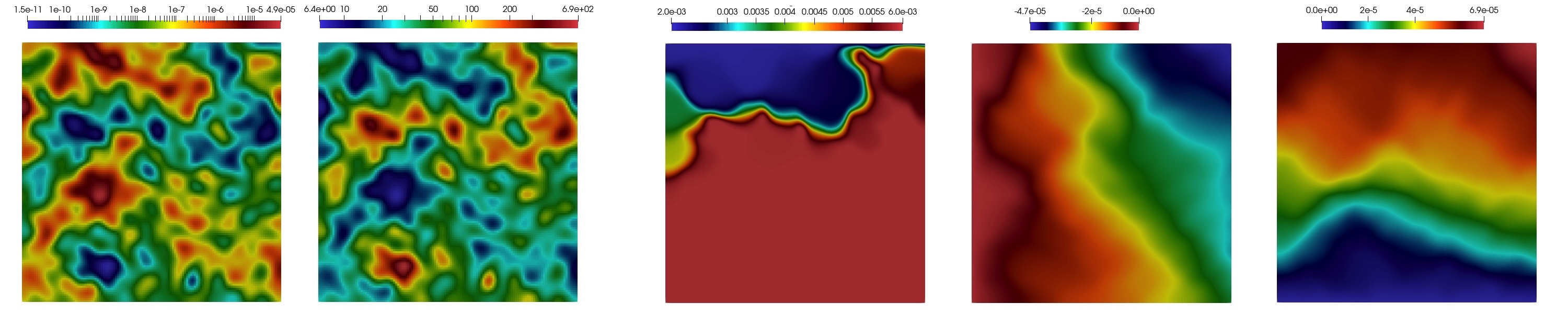}
\caption{Test 2a. Heterogeneous coefficients $k_s(x)$[m$^2$] and $E_d(x)$[MPa], pressure [MPa] and displacements $x$ and $y$  components [m] at final time (from left to right).} 
\label{results2a}
\end{figure}

\begin{figure}[h!]
\centering
\includegraphics[width=1\linewidth]{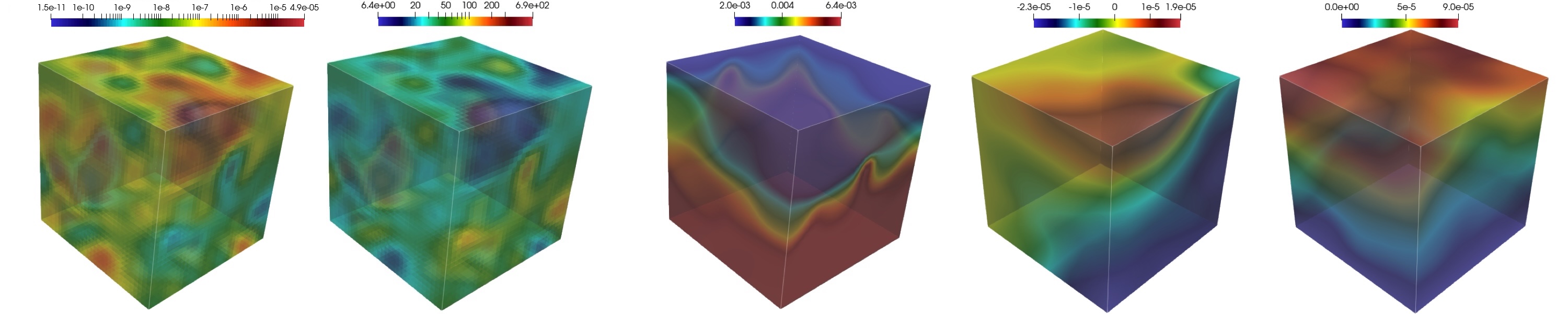}
\caption{Test 2b (3d). Heterogeneous coefficients $k_s(x)$[m$^2$] and $E_d(x)$[MPa], pressure [MPa] and displacements $x$ and $y$  components [m] at final time (from left to right). } 
\label{results3}
\end{figure}

The model parameters are specified as follows:
\begin{itemize}
\item The van Genuchten model \eqref{eq:vg1}--\eqref{eq:vg2} is used to describe the hydraulic nonlinearity, with parameters 
$\nu = 0.37$, $\theta_r = 0.03$, $\theta_s = 0.45$, $\beta = 1$ [1/m], and $n_{\theta} = 1.6$, where $m_{\theta} = 1 - 1/n_{\theta}$.
These values correspond to silt-type soils \cite{siltecho2015use, lu2014power}. For the nonlinear elastic model \eqref{eq:vg3}, we set $r_E = E_d / E_w = 2$ and $\zeta = 1.5$ \cite{lu2014power}.

\item The porosity is set to $\phi = 0.45$, the Biot coefficient is set to $\alpha = 0.2$ and the compressibilities of water and solid grains are taken as $C_w = C_s = 10^{-3}$ [1/MPa]. The water density is $\rho_w = 1000$ [kg/m$^3$], and the gravitational acceleration is $g = 9.8$ [m/s$^2$].

\item For the flow problem, a Robin-type boundary condition is imposed on the top boundary with 
$p_1 = 0.0020286$ [MPa] and $\gamma = 100$, while no-flux conditions are applied on the remaining boundaries.  
For the mechanical problem, zero displacement is enforced in the $x$-direction on the left boundary and in the $y$-direction on the bottom boundary, whereas the remaining boundaries are assumed to be stress-free.

\item The initial pressure is set to $p_0 = 0.006027$ [MPa]. Simulations are performed over the time interval $t \in [0, T_{\max}]$ with $T_{\max} = 2.31$ days. 
\end{itemize}
In Figures~\ref{results2a} and~\ref{results3}, we depict the solutions for \textit{Test 2a} and \textit{Test 2b} using the parameters described above.

\subsubsection{Test 2a (model parameters variation)}

In \textit{Test 2a}, we examine the effects of heterogeneity, contrast, nonlinear parameters, and boundary conditions on the performance of the method using a fixed $128 \times 128$ grid.  
Taking \textit{Test 2a} as a baseline, Figure~\ref{results2b} presents two additional tests.

\begin{figure}[h!]
\centering
\begin{minipage}{0.49\linewidth}
\centering
\includegraphics[width=\linewidth]{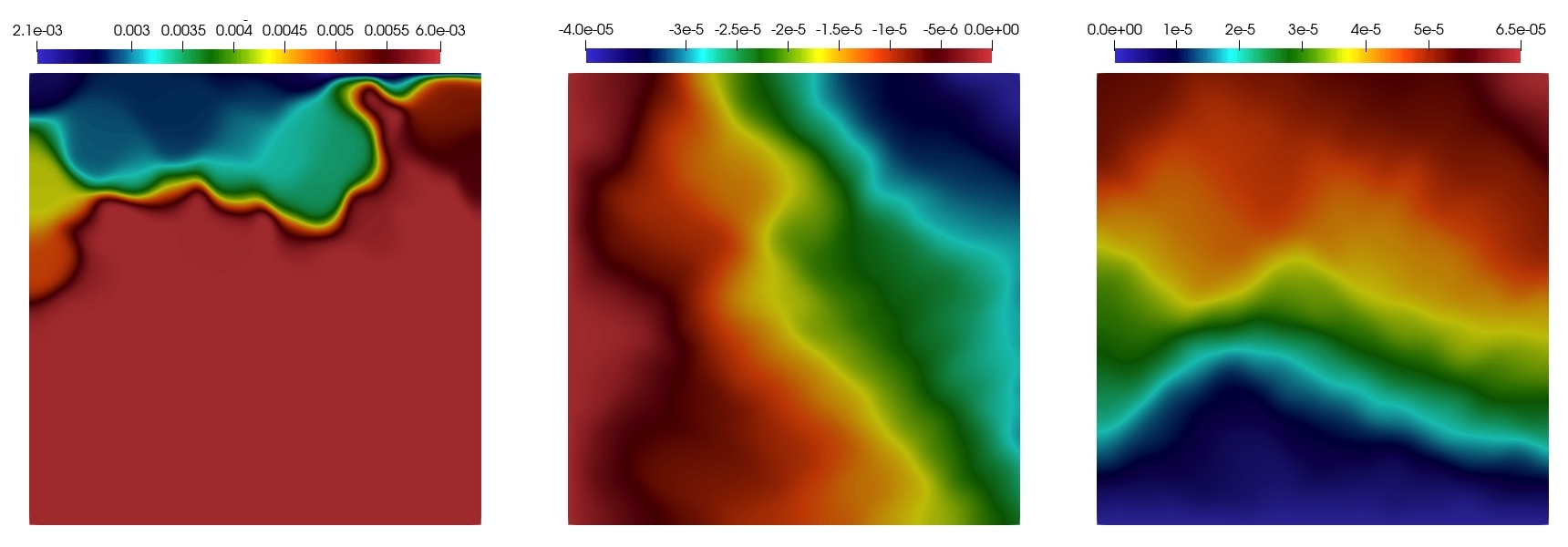}\\
\text{\textit{\footnotesize Test 2a-bc (boundary conditions).}}
\end{minipage}
\hfill
\begin{minipage}{0.49\linewidth}
\centering
\includegraphics[width=\linewidth]{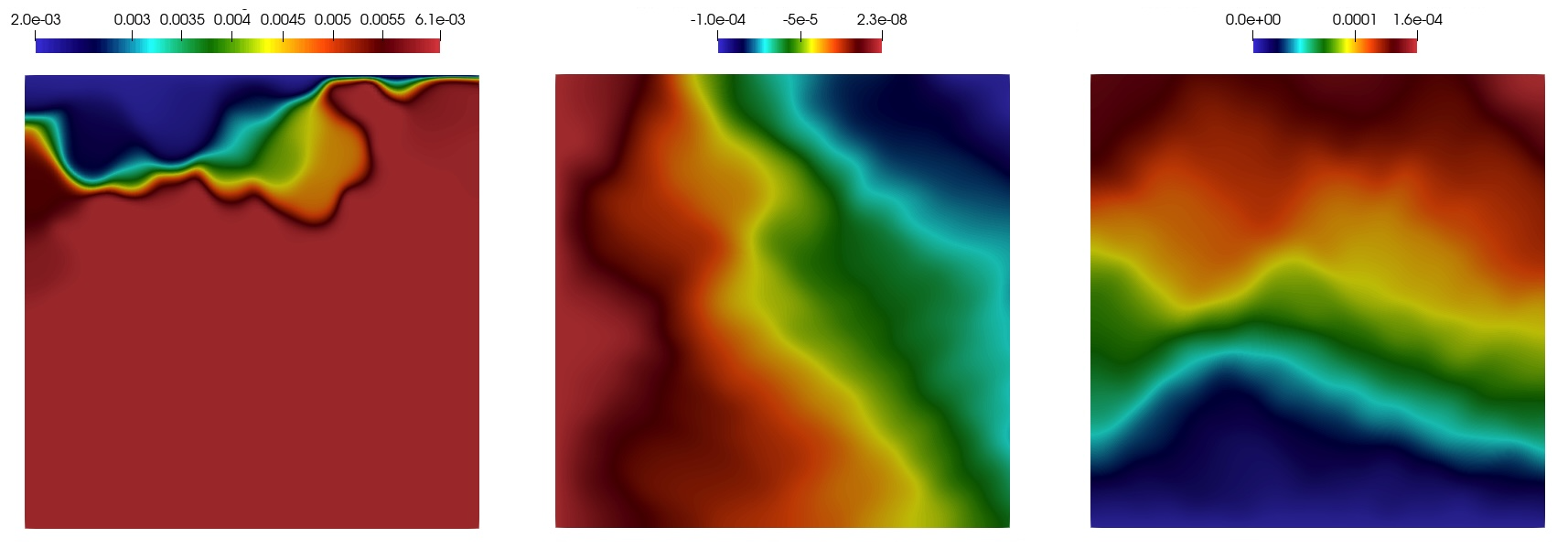}\\
\text{\textit{\footnotesize Test 2a-nl (nonlinear model parameters).}}
\end{minipage}
\caption{Pressure [MPa] and displacement components $u_x$ and $u_y$~[m] at the final time
(from left to right).}
\label{results2b}
\end{figure}

\begin{figure}[h!]
\centering
\includegraphics[width=1\linewidth]{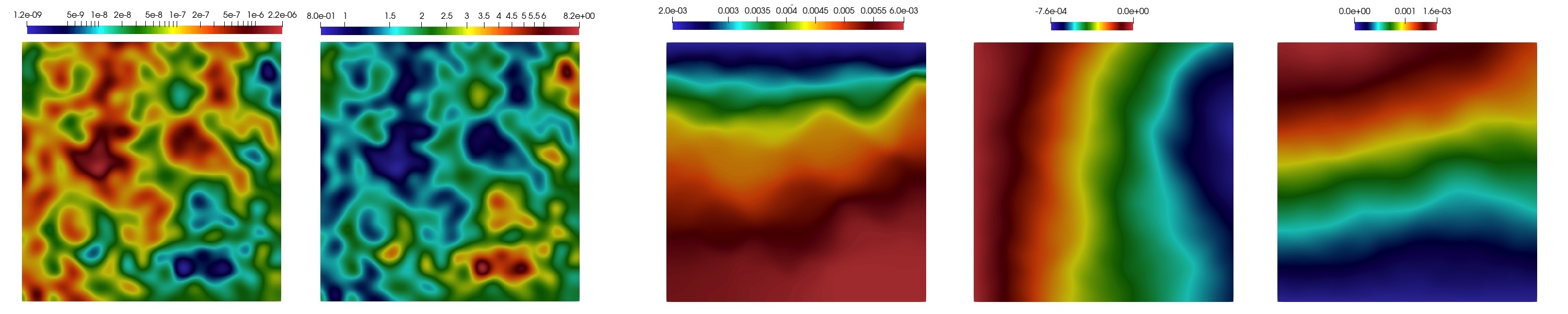}
\caption{Test 2a-c (contrast).  Heterogeneous coefficients $k_s(x)$[m$^2$] and $E_d(x)$[MPa], pressure [MPa] and displacements $x$ and $y$  components [m] at final time (from left to right). } 
\label{results2c}
\end{figure}

\begin{figure}[h!]
\centering
\includegraphics[width=0.24\linewidth]{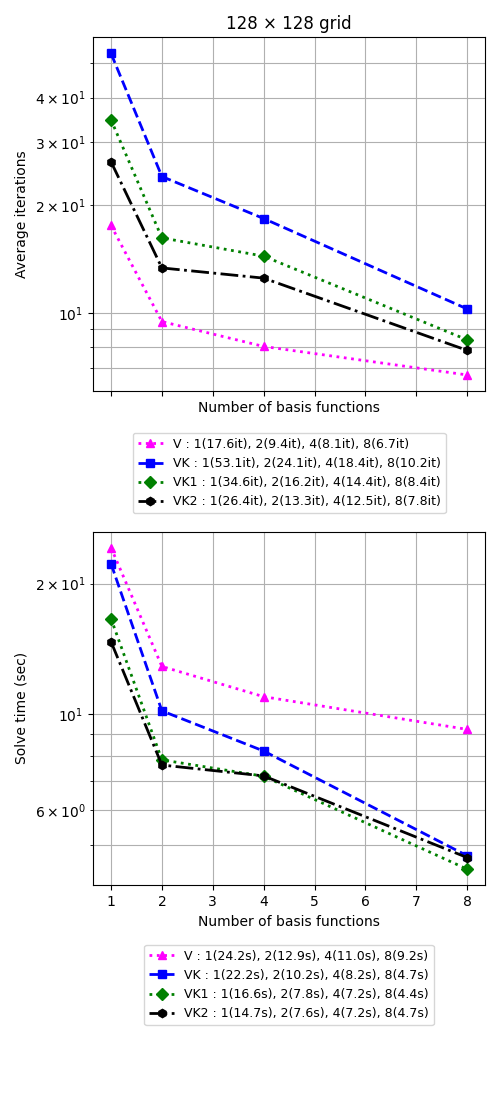}
\includegraphics[width=0.24\linewidth]{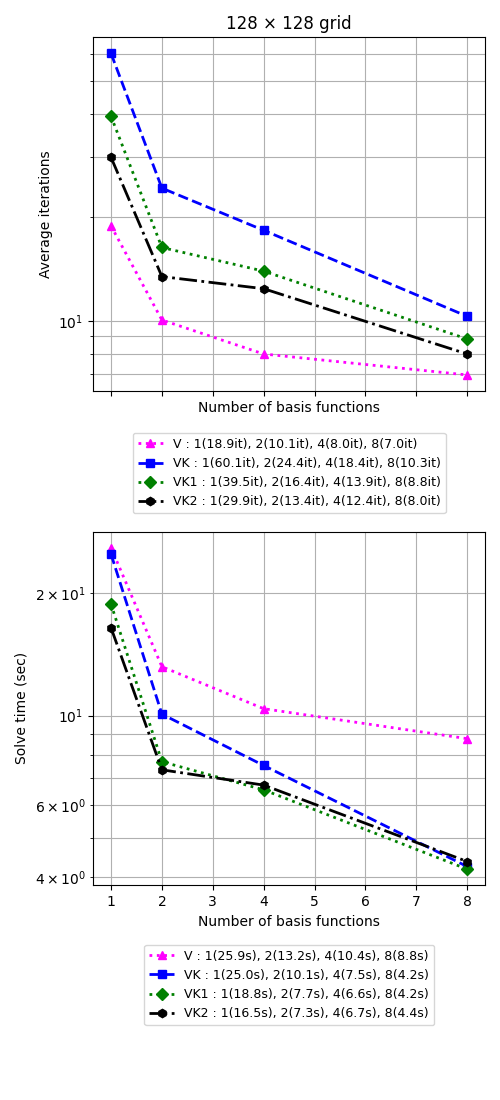}
\includegraphics[width=0.24\linewidth]{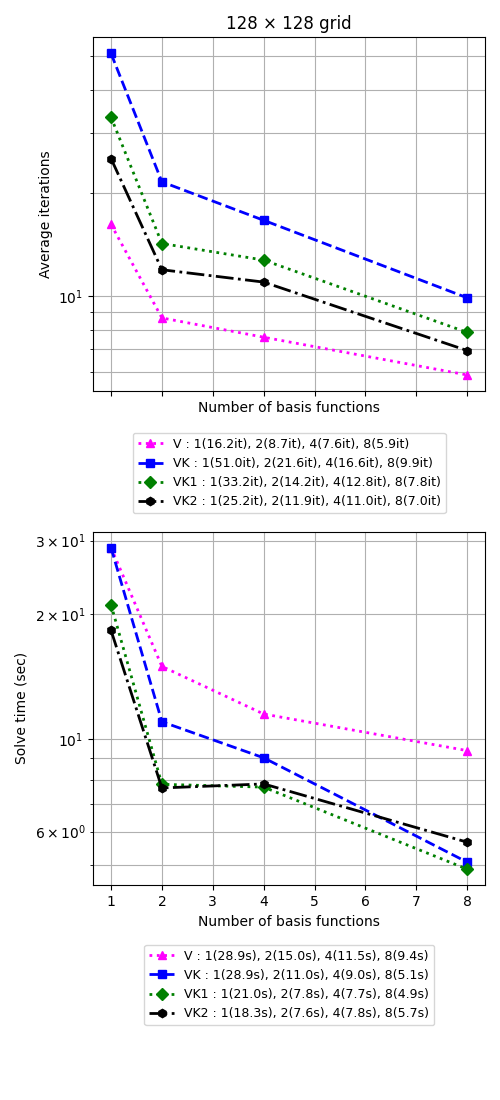}
\includegraphics[width=0.24\linewidth]{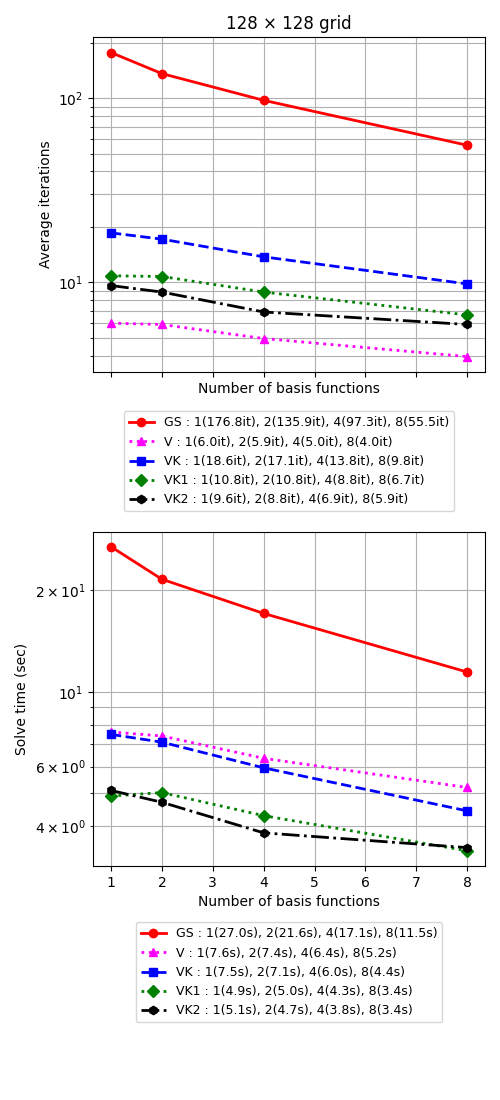}
\caption{Tests 2a, 2a-bc, 2a-nl and 2a-c (from left to right). Average number of iterations of linear solver (first row) and solve time (second row) for a multiscale two-grid solver with different types of smoothers for  $128 \times 128$ grids.  (from left to right). GS: Gauss-Seidel. V: $H/2$ overlapped subdomains $\omega_i$. VK, VK1 and VK2: minimal (interface), $h$ and $2h$ overlapped subdomains $K_i^o$ with $o=0,1,2$.}
\label{msiter2b}
\end{figure}

In \textit{Test 2a-bc}, we reduce the boundary coefficient to $\gamma = 0.1$ to investigate the influence of boundary conditions by varying the parameter associated with the Robin-type boundary condition. The default value $\gamma = 100$ in \eqref{eq:mm5} corresponds to Dirichlet-type conditions, whereas the smaller value $\gamma = 0.1$ represents flux-related (Neumann-type) boundary conditions.  
In \textit{Test 2a-nl}, we study the influence of nonlinear parameters on the solution and the performance of the two-grid method. The default setting corresponds to silt-type rocks, while in this additional test we consider a nonlinear model corresponding to sand. In the van Genuchten model \eqref{eq:vg1}--\eqref{eq:vg2}, the parameters are set as $\nu = 0.5$, $\theta_r = 0.018$, $\theta_s = 0.39$, $\beta = 0.035$ [1/m], and $n_{\theta} = 2.52$, with $m_{\theta} = 1 - 1/n_{\theta}$ \cite{siltecho2015use, lu2014power}. For the nonlinear elastic model \eqref{eq:vg3}, we set $r_E = E_d / E_w = 2$ and $\zeta = 0.2$ \cite{lu2014power}. Additionally, the parameters $\alpha = 0.8$ and $\phi = 0.39$ are used.
Furthermore, we investigate the effect of contrast on the convergence of the method in \textit{Test 2a-c} (see Figure~\ref{results2c}). 

Figure~\ref{msiter2b} summarizes the robustness and efficiency of the proposed multiscale two-grid solver under variations in boundary conditions, nonlinear material parameters, and coefficient contrast. 
Across all variants of \textit{Test~2a}, the iteration  counts remain uniformly bounded, demonstrating that the two-grid method is robust with respect to changes in boundary conditions (\textit{Test~2a-bc}) as well as increased nonlinearities associated with sand-type parameters (\textit{Test~2a-nl}). 
In particular, the coupled overlapping Schwarz/Vanka-type smoothers with enlarged overlap (VK1 and VK2) consistently achieve the smallest number of iterations. Although these smoothers incur a higher per-iteration cost, the total solve time remains competitive and is often reduced due to the significantly lower iteration counts. 
In \textit{Test~2a-c}, decreasing the coefficient contrast leads to a reduction in the number of iterations, and solvers with colored oversampled smoothers preserve fast convergence. Moreover, for this lower-contrast setting, convergence is also observed for the simple pointwise smoother (GS). Overall, the results confirm that the proposed two-grid framework exhibits robust performance with respect to both coefficient contrast and nonlinear effects, achieving a favorable balance between iteration counts and wall-clock time when appropriate overlapping smoothers are employed.

\subsubsection{Test 2b (computational cost)}

Finally, we compare the performance of the proposed two-grid method for both two- and three-dimensional problems (see Figures~\ref{results2a} and~\ref{results3}). In addition, we consider a coarse space constructed using standard linear basis functions.  Specifically, we examine two-dimensional problems on $128 \times 128$ and $256 \times 256$ grids in \textit{Test~2a}, as well as a three-dimensional problem on a $32 \times 32 \times 32$ grid in \textit{Test~2b}.

\begin{figure}[h!]
\centering
\includegraphics[width=0.32\linewidth]{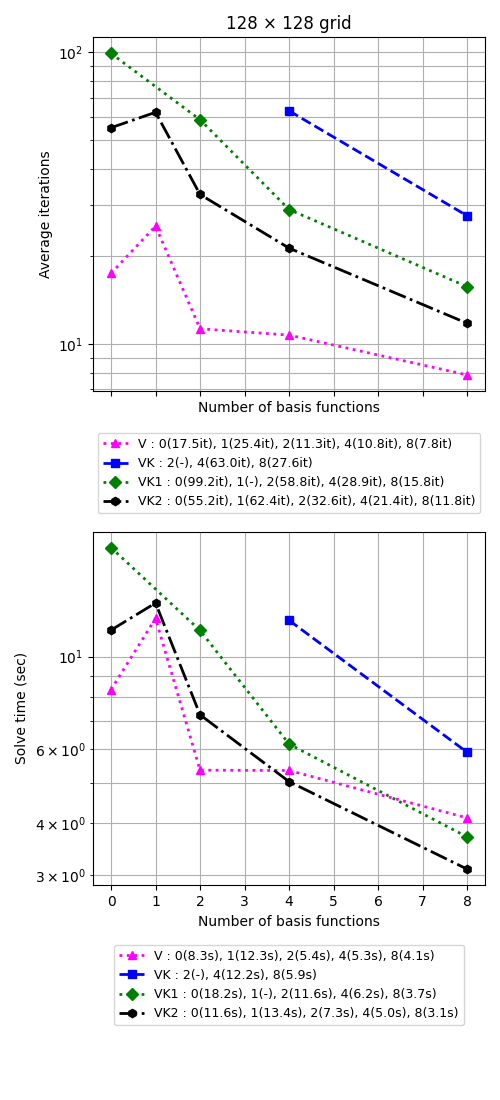}
\includegraphics[width=0.32\linewidth]{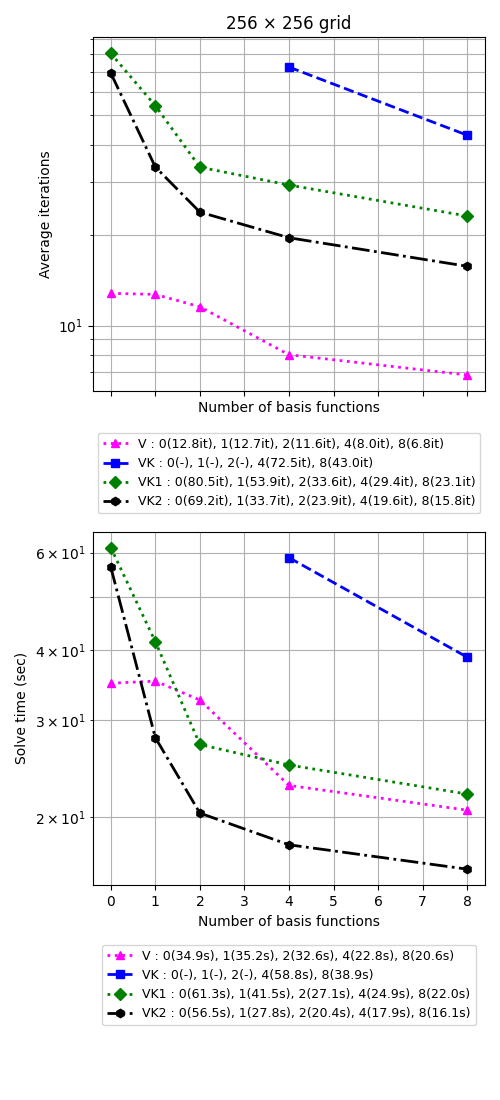}
\includegraphics[width=0.32\linewidth]{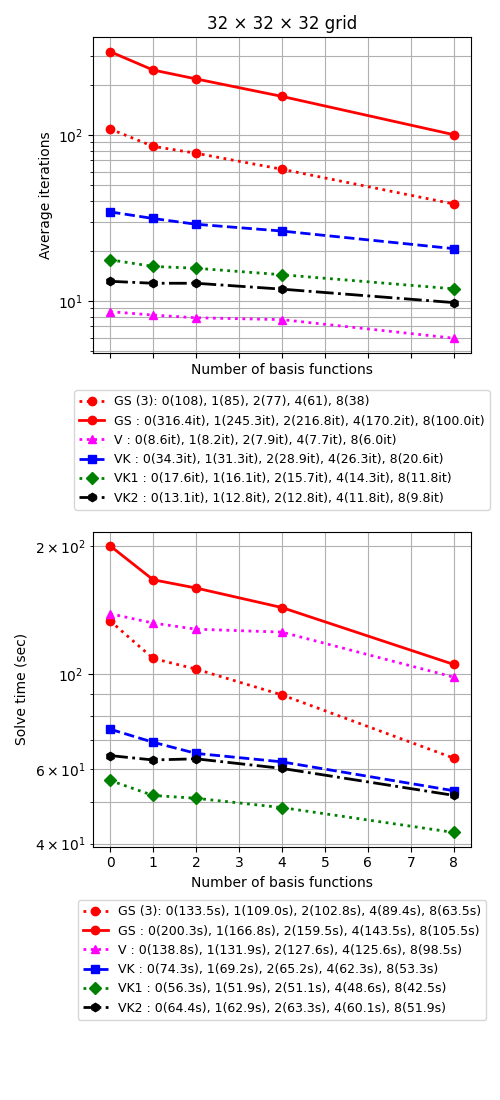}
\caption{Tests 2a on $128^2$ grid, 2a on $256^2$ grid,  and 2b on $32^3$ grid (from left to right). Average number of iterations of linear solver (first row) and solve time (second row) for a multiscale two-grid solver with different types of smoothers and 1 post-smoothing iteration (from left to right). GS: Gauss-Seidel. V: $H/2$ overlapped subdomains $\omega_i$. VK, VK1 and VK2: minimal (interface), $h$ and $2h$ overlapped subdomains $K_i^o$ with $o=0,1,2$. GS(3): Gauss-Seidel with 3 post-smoothing iterations. }
\label{msiter3}
\end{figure}

\begin{table}[h!]
\caption{Tests 2a on $128^2$ grid, 2a on $256^2$ grid, and 2b on $32^3$ grid (from left to right). 
Computational time of the offline stage (eight eigenvectors per local domain).  
$tm^{av}_{loc}$, $DOF^{av}_{loc}$, and $tm^{tot}_{off}$ denote the average 
time (seconds) for solving the local eigenvalue problem (basis construction), the 
average size of the local problem, and the total time calculated as the 
sum over all local domains. }
\label{tab:time-off}
\centering  
\setlength{\tabcolsep}{4pt}
{\footnotesize 	 	
\begin{tabular}{c|ccc|ccc|ccc}
& \multicolumn{3}{c|}{$128 \times 128$} 
& \multicolumn{3}{|c|}{$256 \times 256$}
& \multicolumn{3}{|c}{$32 \times 32 \times 32$} 
\\
 & $tm^{av}_{loc}$ & $DOF^{av}_{loc}$ & $tm^{tot}_{off}$
 & $tm^{av}_{loc}$ & $DOF^{av}_{loc}$ & $tm^{tot}_{off}$
 & $tm^{av}_{loc}$ & $DOF^{av}_{loc}$ & $tm^{tot}_{off}$
 \\
\hline
$p$ 
& 0.045 & 866.9 & 3.633
& 0.259 & 3351.12 & 21.024
& 0.507 & 2628.07 & 63.430 \\
$u$ 
& 0.230 & 1733.9 & 18.600
& 1.936 & 6702.25 & 156.887
& 9.307 & 7884.22 & 1163.38\\
\end{tabular}
}
\end{table}

\begin{table}[h!]
\caption{Tests 2a on $128^2$ grid, 2a on $256^2$ grid,  and 2b on $32^3$ grid (from left to right). Computational time and average size of local problem ($\bar{it}$): $tm_{in}$, $tm_{sys}$, and $tm_{sol}$ denote the cumulative wall-clock times (seconds) for initialization, linear system generation, and linear system solution. The symbol ``--'' indicates that the method did not converge.} 
\label{tab:time-on}
\centering   
\setlength{\tabcolsep}{4pt}
{\footnotesize 	
\begin{tabular}{cc|cccc|cccc|cccc}
& 
& \multicolumn{4}{c|}{$128 \times 128$} 
& \multicolumn{4}{|c|}{$256 \times 256$}
& \multicolumn{4}{|c}{$32 \times 32 \times 32$} 
\\
 &  
 & $tm_{in}$ & $tm_{sys}$ & $tm_{sol}$ & $\bar{it}$ 
 & $tm_{in}$ & $tm_{sys}$ & $tm_{sol}$ & $\bar{it}$ 
 & $tm_{in}$ & $tm_{sys}$ & $tm_{sol}$ & $\bar{it}$  \\
\hline
\multicolumn{14}{c}{sIm}\\
\hline
\multicolumn{2}{c|}{biCG+mg}
& - & 18.38 & 54.07 & 213.0 
& - & 69.15 & 399.87 & 419.0 
& - & 2650.7 & 191.97 & 86.0 \\
\hline
\multicolumn{14}{c}{ImEx}\\
\hline
\multicolumn{2}{c|}{biCG+mg}
& 0.20 & 9.63 & 50.93 & 214.0 
& 0.74 & 41.43 & 404.55 & 415.0 
& 2.41 & 749.22  & 184.00 & 83.0 \\
\hline
\multicolumn{14}{c}{MsImEx, M=0}\\
\hline
GS & 1 & - & - & - & - 
&  - & - & - & - 
&2.36 & 713.83 & 200.29 & 316.4 \\
 & 2 & - & - & - & -
&  - & - & - & - 
&2.39 & 722.16 & 149.37 & 160.2 \\
 & 3 & - & - & - & - 
&  - & - & - & - 
&2.36 & 721.87 & 133.49 & 108.3 \\
\hline
{V} & 1 & 0.81 & 11.29 & 8.34 & 17.5 
&5.80 & 43.24 & 34.90 & 12.8 
&154.05 & 723.80 & 138.76 & 8.6 \\
& 2 & 0.81 & 11.37 & 14.41 & 15.6 
&5.55 & 43.36 & 63.85 & 11.8 
&154.64 & 722.26 & 215.59 & 6.9 \\
& 3 & 0.67 & 11.48 & 20.06 & 14.6 
&5.63 & 43.61 & 92.99 & 11.3 
&154.38 & 722.70 & 318.63 & 6.8 \\
\hline
{VK} & 1 & - & - & - & - 
& - & - & - & - 
&6.20 & 722.53 & 74.30 & 34.3 \\
& 2 & 0.32 & 11.39 & 18.09 & 62.8 
&1.35 & 43.23 & 100.62 & 78.9 
&6.12 & 722.17 & 78.04 & 19.6 \\
& 3 & 0.31 & 11.29 & 21.65 & 53.0 
&1.30 & 43.69 & 119.75 & 63.7 
&6.18 & 723.48 & 85.83 & 14.8 \\
\hline
{VK1} & 1 & 0.36 & 11.46 & 18.23 & 99.2 
&1.26 & 43.23 & 61.27 & 80.5 
&11.97 & 723.15 & 56.31 & 17.6 \\
& 2 & 0.36 & 11.32 & 16.61 & 50.4 
&1.44 & 43.38 & 84.31 & 60.2 
&11.81 & 723.16 & 71.11 & 11.8 \\
& 3 & 0.35 & 11.49 & 21.60 & 45.6 
&1.44 & 43.23 & 109.85 & 54.0 
&11.87 & 722.89 & 95.78 & 10.8 \\
\hline
{VK2} & 1 & 0.40 & 11.42 & 11.58 & 55.2 
&1.63 & 43.02 & 56.51 & 69.2 
&25.37 & 723.06 & 64.42 & 13.1 \\
& 2 & 0.40 & 11.31 & 17.24 & 45.2 
&1.59 & 43.57 & 84.13 & 55.0 
&25.41 & 722.63 & 93.48 & 9.8 \\
& 3 & 0.41 & 11.51 & 22.09 & 39.8 
&1.56 & 43.09 & 111.49 & 49.6 
&25.25 & 723.34 & 124.75 & 8.9 \\
\hline
\multicolumn{14}{c}{MsImEx, M=8}\\
\hline
{GS} & 1 & - & - & - & - 
&- & - & - & - 
&-     & -     & -     & - \\
& 2 & - & - & - & -
&- & - & - & - 
&6.33  & 719.73 & 72.72 & 53.4 \\
& 3 & - & - & - & -
&- & - & - & - 
&6.30  & 714.10 & 63.54 & 38.3 \\
\hline
{V} & 1 & 1.11 & 11.00 & 4.12 & 7.8 
&6.90 & 44.03 & 20.61 & 6.8 
&156.57 & 714.33 & 98.46  & 6.0 \\
& 2 & 1.08 & 10.98 & 6.54 & 6.8 
&6.74 & 43.45 & 32.65 & 5.9 
&155.94 & 714.33 & 180.54 & 5.7 \\
& 3 & 1.07 & 11.03 & 9.22 & 6.7 
&6.55 & 43.23 & 47.93 & 5.8 
&155.94 & 714.70 & 231.91 & 5.0 \\
\hline
{VK} & 1 & 0.58 & 10.91 & 5.91 & 27.6 
&2.21 & 43.13 & 38.89 & 43.0 
&10.17 & 714.12 & 53.26 & 20.6 \\
& 2 & 0.55 & 11.07 & 4.50 & 13.2 
&2.31 & 42.99 & 34.69 & 23.9 
&10.05 & 714.49 & 54.05 & 12.2 \\
& 3 & 0.57 & 10.98 & 4.71 & 10.2 
&2.23 & 43.36 & 33.32 & 16.9 
&10.04 & 715.05 & 60.99 & 9.8 \\
\hline
{VK1} & 1 & 0.61 & 10.91 & 3.71 & 15.8 
&2.38 & 42.63 & 22.04 & 23.1 
&15.96 & 714.49 & 42.51 & 11.8 \\
& 2 & 0.60 & 10.97 & 3.54 & 9.4 
&2.31 & 42.54 & 19.93 & 12.9 
&15.78 & 714.64 & 56.95 & 8.8 \\
& 3 & 0.61 & 11.01 & 4.39 & 8.4 
&2.38 & 43.22 & 21.45 & 9.8 
&15.77 & 714.29 & 73.16 & 7.9 \\
\hline
{VK2} & 1 & 0.65 & 11.01 & 3.10 & 11.8 
&2.44 & 42.60 & 16.14 & 15.8 
&29.09 & 713.79 & 51.95 & 9.8 \\
& 2 & 0.63 & 11.08 & 3.62 & 8.4 
&2.41 & 42.02 & 15.40 & 9.2 
&29.10 & 714.44 & 76.10 & 7.8 \\
& 3 & 0.66 & 11.04 & 4.67 & 7.8 
&2.46 & 43.00 & 18.47 & 7.8 
&29.02 & 714.36 & 98.36 & 6.8 \\
\end{tabular}
}
\end{table}

Figure~\ref{msiter3} compares the performance of the proposed multiscale two-grid solver for both two- and three-dimensional problems. The results indicate that simple pointwise Gauss–Seidel smoothing is insufficient for robust convergence in the two-dimensional case and is effective only for relatively small three-dimensional problems. In contrast, overlapping Schwarz and Vanka-type smoothers significantly enhance robustness across all considered test cases. In particular, the multiscale Vanka smoothers with enlarged overlap (VK1 and VK2) consistently achieve the smallest iteration counts across all tested configurations, leading to a reduction in overall solve time due to faster convergence. The advantages of increased overlap become more pronounced in three dimensions, where minimal or pointwise smoothers result in a substantial growth in iteration counts, whereas oversampled multiscale smoothers maintain stable and efficient convergence. These results demonstrate that an appropriate choice of coupled overlapping smoothers is essential for achieving scalable performance of the proposed multiscale two-grid framework.

Tables~\ref{tab:time-off} and \ref{tab:time-on} report offline (precomputation) and online costs. The offline stage is dominated by solving the local generalized eigenvalue problems \eqref{eq:evp-p} and \eqref{eq:evp-u} for pressure and displacement; Table~\ref{tab:time-off} (eight eigenvectors per subdomain) lists the average basis-construction time ($tm^{av}{loc}$), average local problem size ($DOF^{av}{loc}$), and the total offline time summed over subdomains ($tm^{tot}_{off}$). Because the local eigenproblems are independent, these computations are naturally parallelizable.

Table~\ref{tab:time-on} compares online performance in 2D and 3D for standard linearized implicit (sIm), implicit–explicit (ImEx), and multiscale implicit–explicit (MsImEx) schemes with various solvers/smoothers. For sIm, Krylov solves (BiCGStab in scipy \cite{virtanen2020scipy}) with smoothed-aggregation AMG preconditioning (pyamg \cite{bell2023pyamg}) require many iterations and long solve times, reflecting the difficulty of monolithic AMG solvers and resolving nonlinear coupling. ImEx reduces per-step time by fixing the matrix and no longer having to resolve the nonlinearity, but Krylov iteration counts remain similar. In contrast, MsImEx substantially lowers iteration counts and solve times by using an enriched multiscale coarse space with overlapping Schwarz and Vanka-type smoothers. Although MsImEx adds one-time offline setup (multiscale bases and local block solvers), these costs are amortized across nonlinear time steps since the precomputed components are reused. We also report the non-enriched coarse space ($M=0$), as a loose proxy for geometric multigrid, which improves over sIm/ImEx but is less effective than spectral enrichment. Overall, once the offline stage is complete, MsImEx consistently converges in only a few iterations and delivers the lowest per-step solve times, especially for larger problems.

\section{Conclusion}\label{sec:con}

In this work, we consider a coupled nonlinear system of equations that describe an unsaturated flow problem in heterogeneous poroelastic media. We derive a standard implicit scheme for the coupled system and then propose an implicit-explicit formulation, in which the stiff components of the problem are represented by a single linear operator that is treated implicitly. Careful construction of the additive splitting approach is presented for the unsaturated poroelasticity problem with the van Genuchten nonlinearity model and power-low for elastic modulus. We show that the proposed implicit-explicit scheme is unconditionally stable in the linear stability sense and enables an efficient implementation. Next, based on the proposed linearization, we develop a multiscale two-grid solver with coupled overlapping Schwarz smoothing iterations. Accurate coarse grid approximation is constructed based on an enrichment of the linear coarse space with local spectral basis functions for both pressure and elasticity equations. A coupled smoothing iteration is proposed based on the overlapping Vanka approach and defining blocks relative to a coarse grid (oversampled coarse cell and local support of multiscale basis functions). Although the block smoothers present a significant computational cost, similar to the multiscale basis functions, by combining with the implicit-explicit formulation we are able to precompute these quantities once ``offline'' for a full nonlinear simulation. A numerical study is presented for nonlinear coupled test problem of unsaturated flow in heterogeneous poroelastic media.

\section*{Acknowledgments}
BSS was supported by the DOE Office of Advanced Scientific Computing Research Applied Mathematics program through Contract No. 89233218CNA000001. Los Alamos National Laboratory report number LA-UR-25-25055.

\bibliographystyle{unsrt}
\bibliography{lit}

\appendix

\section{Details on the mathematical model}\label{app1}

We have a mass conservation equation in terms of $(S, p)$ \cite{kim2010sequential,varela2018implementation, varela2021finite}
\begin{equation}
\label{eq:mconsr}
\begin{split}
\rho_w f & = \frac{\partial}{\partial t} (\rho_w \phi S) + \nabla \cdot (\rho_w \phi S q_w)\\
& \approx \phi S \frac{\partial \rho_w}{\partial t}
+ \rho_w S \frac{\partial \phi}{\partial t}
+ \phi \rho_w \frac{\partial S}{\partial t}
+
\rho_w \nabla \cdot (\phi S q_{ws}) 
+ \rho_w \phi S  \nabla \cdot q_s \\
& = \rho_w \Big( 
\phi S C_w \frac{\partial p}{\partial t}
+ \phi \frac{\partial S}{\partial t}
+ \nabla \cdot q_w  
\Big)
+ \rho_w S \Big( 
\frac{\partial \phi}{\partial t} + \phi \nabla \cdot q_s
\Big),
\end{split}
\end{equation}
where $C_w$ is the water compressibility, $q_{ws} = q_w - q_s$ is the velocity of the water with respect to the solids and $q_w = \phi S q_{ws}$ is the Darcy velocity of the water phase, 
\[
q_w = - k(x) \frac{k_{rw}(S)}{\mu_w} \nabla(p - \rho_w g),
\]
Here $k$ is the intrinsic permeability tensor, $\mu_w$ is the water dynamic viscosity, and $k_{rw}$ is the relative permeability that accounts for the simultaneous flow of water and air. In hydrology, Darcy’s law is often expressed in terms of the pressure, where $h_w = (p - p_a)/(\rho_w g)$ is the water pressure head (relative to the atmospheric pressure $p_a$) and $k_w= \rho_w g \ k/\mu_w$ is the hydraulic conductivity at saturated conditions. Under Richards' assumption ($p_a=0$), we obtain eq. \eqref{eq:unsat}.

For the conservation of the solid phase, we have \cite{varela2018implementation}
\[
\begin{split}
0 &= \frac{\partial}{\partial t} \big((1-\phi) \rho_s\big) 
+ \nabla \cdot \big((1-\phi) \rho_s q_s\big) 
\approx 
(1-\phi) \frac{\partial \rho_s}{\partial t} 
- \rho_s \frac{\partial \phi}{\partial t} 
+ (1 - \phi) \rho_s \nabla \cdot q_s \\
& = 
- \rho_s \Big( \frac{\partial \phi}{\partial t}  + \phi  \nabla \cdot q_s \Big)
+ (1-\phi) \Big( 
    \frac{\partial \rho_s}{\partial \sigma_T} \frac{\partial \sigma_T}{\partial t} 
  + \frac{\partial \rho_s}{\partial p_s} \frac{\partial p_s}{\partial t}
\Big)
+  \rho_s \nabla \cdot q_s.
\end{split}
\]
Here, we have
\[
\frac{\partial \rho_s}{\partial \sigma_T} = -\frac{\rho_s C_s}{1 - \phi}, \quad 
\frac{\partial \rho_s}{\partial p_s} = -\frac{\rho_s C_s \phi}{1 - \phi}, \quad 
\nabla \cdot q_s = \frac{\partial \varepsilon_v}{\partial t}, \quad 
\frac{\partial \sigma_T}{\partial t} 
= \frac{1}{C_m} \nabla \cdot q_s 
- \alpha p \frac{\partial S}{\partial t} 
- \alpha S \frac{\partial p}{\partial t}, 
\]
with $\varepsilon_v = \nabla \cdot u$, $p_s = Sp$ and $\alpha = 1 - C_s/C_m$, where $C_s$ and $C_m$ are the compressibility of the solid grains and porous medium.  
Therefore, we obtain
\begin{equation}
\label{eq:sconsr}
\begin{split}
\frac{\partial \phi}{\partial t}  + \phi  \nabla \cdot q_s 
& = 
- \rho_s \Big( \frac{\partial \phi}{\partial t}  + \phi  \nabla \cdot q_s \Big)
+ (1-\phi) \Big( 
    \frac{\partial \rho_s}{\partial \sigma_T} \frac{\partial \sigma_T}{\partial t} 
  + \frac{\partial \rho_s}{\partial p_s} \frac{\partial p_s}{\partial t}
\Big)
+  \rho_s \nabla \cdot q_s\\
& = 
- C_s \Big( 
   \frac{\partial \sigma_T}{\partial t} 
  + \phi \frac{\partial p_s}{\partial t}
\Big)
+   \frac{\partial \varepsilon_v}{\partial t} \\
& = - C_s \Big( 
   \frac{1}{C_m} \frac{\partial \varepsilon_v}{\partial t}
- \alpha p \frac{\partial S}{\partial t} 
- \alpha S \frac{\partial p}{\partial t}
+ \phi p \frac{\partial S}{\partial t} 
+ \phi S \frac{\partial p}{\partial t}
\Big)
+  \frac{\partial \varepsilon_v}{\partial t}\\
& = \alpha \frac{\partial \varepsilon_v}{\partial t} 
+ C_s \Big( 
(\alpha - \phi) p \frac{\partial S}{\partial t} 
+ (\alpha - \phi) S \frac{\partial p}{\partial t}
\Big)
\end{split}
\end{equation}
By combining \eqref{eq:mconsr} and \eqref{eq:sconsr}, we obtain
\[
\begin{split}
f &= \phi S C_w \frac{\partial p}{\partial t}
+ \phi \frac{\partial S}{\partial t}
+ \nabla \cdot q_w  
+ S \Big( 
\frac{\partial \phi}{\partial t} + \phi \nabla \cdot q_s
\Big) = \\
& = 
\phi S C_w \frac{\partial p}{\partial t}
+ \phi \frac{\partial S}{\partial t}
+ \nabla \cdot q_w  
+ \alpha S \frac{\partial \varepsilon_v}{\partial t} 
+ S C_s \Big( 
(\alpha - \phi) p \frac{\partial S}{\partial t} 
+ (\alpha - \phi) S \frac{\partial p}{\partial t}
\Big) \\ 
& = 
\big( \phi C_w + (\alpha -\phi) C_s S \big) S \frac{\partial p}{\partial t}
+ \big(\phi + (\alpha -\phi) C_s  S p \big) \frac{\partial S}{\partial t}
+ \alpha S \frac{\partial \varepsilon_v}{\partial t} 
+ \nabla \cdot q_w.
\end{split}
\]
Then, we have
\[
c(x, p) \frac{\partial p}{\partial t}
+ \alpha S \frac{\partial \varepsilon_v}{\partial t} 
- \nabla \cdot \Big( \kappa(x, p) \nabla(p - \rho_w g) \Big) = f,
\]
with \[
c(x, p) =  \big( \phi C_w  + (\alpha -\phi) C_s S \big) S + \big(\phi + (\alpha -\phi) C_s  S p \big) S', \quad 
\kappa(x, p) = k(x) \frac{k_{rw}(S)}{\mu_w}, 
\]
where $S = S(p)$ and $S' = \partial S/\partial p$.


\end{document}